\documentclass[reqno, 12pt]{amsart}
\usepackage{a4wide}
\usepackage{amscd}
\usepackage{amsmath}
\usepackage{amsfonts}
\usepackage{amssymb}
\usepackage{multirow}
\usepackage{color}
\input xy
\xyoption{all}

\allowdisplaybreaks

\usepackage[OT2,OT1]{fontenc}
\newcommand\cyr{%
\renewcommand\rmdefault{wncyr}%
\renewcommand\sfdefault{wncyss}%
\renewcommand\encodingdefault{OT2}%
\normalfont
\selectfont}
\DeclareTextFontCommand{\textcyr}{\cyr}

\newcommand{\ca}{{\mathfrak{a}}}

\newcommand{\cf}{{\mathfrak{f}}}

\newcommand{\cg}{{\mathfrak{g}}}

\newcommand{\ch}{{\mathfrak{h}}}

\newcommand{\ck}{{\mathfrak{k}}}

\newcommand{\co}{{\mathfrak{o}}}

\newcommand{\cp}{{\mathfrak{p}}}

\newcommand{\cs}{{\mathfrak{s}}}

\newcommand{\cu}{{\mathfrak{u}}}

\newcommand{\cz}{{\mathfrak{z}}}

\newcommand{\RR}{\mathbb{R}}

\newcommand{\ZZ}{\mathbb{Z}}
\newcommand{\CC}{\mathbb{C}}

\newcommand{\ad}{\mbox{ad}}

\newcommand{\grad}{\mbox{grad}}

\newcommand{\hess}{\mbox{hess}}
\newcommand{\rk}{\mbox{rk}}

\newtheorem{thm}{Theorem}[section]

\newtheorem{prop}[thm]{Proposition}
\newtheorem{cor}[thm]{Corollary}
\newtheorem{lm}[thm]{Lemma}

\numberwithin{equation}{section}

\begin{document}

\title[Real hypersurfaces with isometric Reeb flow]{Real hypersurfaces with isometric Reeb flow\\ in K\"{a}hler manifolds}
\author{\textsc{J\"{u}rgen Berndt} and \textsc{Young Jin Suh}}

\address{King's College London\\ Department of Mathematics \\  London\\  WC2R 2LS \\ United Kingdom}
\email{jurgen.berndt@kcl.ac.uk}
\address{Kyungpook National University \\ Department of Mathematics \\ Daegu 41566 \\Republic of Korea}
\email{yjsuh@knu.ac.kr}
\thanks{This work was supported by grant NRF-2015-R1A2A1A-01002459 from the National Research Foundation of Korea}

\begin{abstract}
We investigate the structure of real hypersurfaces with isometric Reeb flow in K\"{a}hler manifolds. As an application we classify real hypersurfaces with isometric Reeb flow in irreducible Hermitian symmetric spaces of compact type.
\end{abstract}

\maketitle
\thispagestyle{empty}


\section{Introduction}

In this article we investigate the Reeb vector flow on real hypersurfaces in K\"{a}hler manifolds. Let $M$ be a connected orientable real hypersurface in a K\"{a}hler manifold $\bar{M}$ and $\zeta$ be a unit normal vector field on $M$. Denote by $J$ the complex structure on $\bar{M}$. The tangent vector field $\xi = -J\zeta$ on $M$ is the Reeb vector field on $M$ and its flow is the Reeb flow on $M$. The dynamics of the Reeb vector field is an important topic in the context of contact geometry. Here we consider the Reeb vector field in the context of almost contact geometry, which is the geometry that is naturally induced on real hypersurfaces in K\"{a}hler manifolds.

The paper consists of two main parts. In the first part we develop a general structure theory for real hypersurfaces in K\"{a}hler manifolds for which the Reeb flow preserves the induced metric. This condition is of interest for example in the K\"{a}hler reduction construction of minimal Lagrangian submanifolds in K\"{a}hler quotients. We derive some general equations involving geometric objects such as shape operators and curvature tensors. From these equations we will extract some interesting geometric information. For example, if the K\"{a}hler manifold is an irreducible Hermitian symmetric space of compact type, we can deduce that the flow lines of the Reeb flow must be closed curves, which is one of the main questions generally asked in relation to the Reeb flow. 

In the second part we apply the general theory to classify real hypersurfaces with isometric Reeb flow in  irreducible Hermitian symmetric spaces of compact type. We first review the structure theory of Hermitian symmetric spaces of compact type. These particular symmetric spaces correspond, via the Borel-de Siebenthal construction method, to simple roots whose coefficient in the highest root of the root space decomposition of a complex semisimple Lie algebra is equal to one. This allows us, using Chevalley bases, to construct algebraic models of Hermitian symmetric spaces of compact type. With a mixture of algebraic and geometric methods, and using the results of the first part, we obtain the following classification result:

\begin{thm} \label{classification}
Let $M$ be a connected orientable real hypersurface in an irreducible Hermitian symmetric space $\bar{M}$ of compact type. If the Reeb flow on $M$ is an isometric flow, then $M$ is congruent to an open part of a tube of radius $0 < t < \pi/\sqrt{2}$ around the totally geodesic submanifold $\Sigma$ in $\bar{M}$, where
\begin{itemize}
\item[(i)] $\bar{M} = \CC P^r = SU_{r+1}/S(U_1U_r)$ and $\Sigma = \CC P^k$, $0 \leq k \leq r-1$;
\item[(ii)] $\bar{M} = G_k(\CC^{r+1}) = SU_{r+1}/S(U_kU_{r+1-k})$ and $\Sigma = G_k(\CC^r)$, $2 \leq k \leq \frac{r+1}{2}$;
\item[(iii)] $\bar{M} = G_2^+(\RR^{2r}) = SO_{2r}/SO_{2r-2}SO_2$ and $\Sigma = \CC P^{r-1}$, $3 \leq r$;
\item[(iv)] $\bar{M} = SO_{2r}/U_r$ and $\Sigma = SO_{2r-2}/U_{r-1}$, $5 \leq r$.
\end{itemize}
Conversely, the Reeb flow on any of these hypersurfaces is an isometric flow. 
\end{thm}

The underlying Riemannian metric on $\bar{M}$ in Theorem \ref{classification} is the one that is induced naturally by the Killing form of the Lie algebra of the isometry group. As an immediate consequence of Theorem \ref{classification} we obtain the following nonexistence result:

\begin{cor}\label{noexist}
There are no real hypersurfaces with isometric Reeb flow in the Hermitian symmetric spaces $G_2^+(\RR^{2r+1})$ {\rm ($2 \leq r$)}, $Sp_r/U_r$ {\rm ($3 \leq r$)}, $E_6/Spin_{10}U_1$ and $E_7/E_6U_1$.
\end{cor}

In Section \ref{StcHss} we will explain how all irreducible Hermitian symmetric spaces of compact type are constructed from the root systems of type ($A_r$), ($B_r$), ($C_r$), ($D_r$), ($E_6$) and ($E_7$). Corollary \ref{noexist} can be rephrased in terms of root systems as follows:

\begin{cor} \label{rs?}
An irreducible Hermitian symmetric space of compact type admits a real hypersurface with isometric Reeb flow if and only if the underlying root system is of type {\rm ($A_r$)} or {\rm ($D_r$)}.
\end{cor}

We emphasize that this is an observation based on the classification result. We have no direct explanation of this fact and it would be interesting to find a direct argument.

All the real hypersurfaces in Theorem \ref{classification} are orbits of cohomogeneity one actions and therefore homogeneous. This gives as another consequence:

\begin{cor}
Any connected orientable real hypersurface $M$ with isometric Reeb flow in an irreducible Hermitian symmetric space of compact type is locally homogeneous. In particular, if $M$ is complete, then $M$ is homogeneous. 
\end{cor}

It is remarkable that the existence of this one-parameter group of isometries implies (local) homogeneity and therefore has such a strong influence on the geometry of $M$. 

We give a brief sketch of the proof. Using methods from Riemannian geometry and the recently developed theory about the index of symmetric spaces, we first show that $M$ has constant principal curvatures and that the geometry of $M$ is in some sense nicely adapted to the geometry of the ambient space $\bar{M}$. We then reformulate this is in algebraic terms using algebraic models for the Hermitian symmetric spaces of compact type that are based on Borel-de Siebenthal theory. Using Lie algebraic methods we show that the normal vectors of $M$ correspond to the highest root in the algebraic model. The geometric interpretation of this is that the geodesics in $\bar{M}$ that intersect $M$ perpendicularly are closed geodesics in $\bar{M}$ of shortest length. We then apply Jacobi field theory to prove that $M$ has two totally geodesic complex focal sets. Rigidity of totally geodesic submanifolds tells us that $M$ must lie on a tube around each of its two focal sets. The problem we are now facing is that totally geodesic submanifolds in Hermitian symmetric spaces are not classified (except for rank $1$ and $2$), and so we need further arguments to determine the structure of the focal sets. This is finally achieved by investigating representations of certain compact Lie algebras. 

\section{Structure theory of real hypersurfaces with isometric Reeb flow} \label{StorhwiRf}

In this section we investigate the geometry of real hypersurfaces with isometric Reeb flow in K\"{a}hler manifolds and, in some more depth, in Hermitian symmetric spaces. For general background on K\"{a}hler manifolds, including the curvature identities we are using in this section, we refer the reader to \cite{Ba06}. A good resource for Hermitian symmetric spaces is \cite{He01}.

Let $\bar{M}$ be a K\"{a}hler manifold. We denote by $g$ the Riemannian metric, by $J$ the K\"{a}hler structure, by $\bar\nabla$ the Levi Civita covariant derivative and by $\bar{R}$ the Riemannian curvature tensor on $\bar{M}$.

Let $M$ be a connected orientable real hypersurface in $\bar{M}$. We denote also by $g$ the Riemannian metric on $M$ that is induced from the one on $\bar{M}$. Since $M$ is orientable, there exists a global unit normal vector field $\zeta$ on $M$. The vector field $\xi = -J\zeta$ is the so-called Reeb vector field on $M$. The one-form on $M$ dual to $\xi$ is denote by $\eta$, that is, $\eta(X) = g(X,\xi)$, $X \in TM$.  The structure tensor $\phi$ on $M$ is defined by $\phi X = JX - g(JX,\zeta)\zeta = JX -  \eta(X)\zeta$. We denote by $\nabla$ the Levi Civita covariant derivative on $M$ and by $S$ the shape operator of $M$ with respect to $\zeta$. The normal Jacobi operator $\bar{R}_\zeta$ on $M$ is defined by $\bar{R}_\zeta X = \bar{R}(X,\zeta)\zeta$. We denote by ${\mathcal C} = \ker(\eta)$ the maximal complex subbundle of the tangent bundle $TM$ of $M$  and by $X_{\mathcal C}$ the orthogonal projection of $X \in T_p\bar{M}$, $p \in M$, onto ${\mathcal C}_p$. 

We first develop some general theory for real hypersurfaces for which the Reeb flow is a geodesic flow, that is, each integral curve of $\xi$ is a geodesic in $M$. Such hypersurfaces are also known in the literature as Hopf hypersurfaces. We start with a simple but very useful observation. 

\begin{prop}\label{geodReeb}
A connected orientable real hypersurface $M$ in a K\"{a}hler manifold $\bar{M}$ has geodesic Reeb flow if and only if $\xi$ is a principal curvature vector of $M$ everywhere.
\end{prop}

\begin{proof} The Reeb flow on $M$ is a geodesic flow if and only if $\nabla_\xi \xi = 0$. The Gau{\ss} and Weingarten formulae imply
$g(\nabla_\xi \xi, X) = -g(\bar\nabla_\xi J\zeta,X) = -g(J\bar\nabla_\xi\zeta,X) = g(JS\xi,X) = g(\phi S\xi,X)$
for all vector fields $X$ on $M$. It follows that $\nabla_\xi \xi = 0$ if and only if $\phi S\xi = 0$. The latter equation means that $S\xi \in \ker(\phi) = \RR \xi$. 
\end{proof}

If $\xi$ is a principal curvature vector of $M$ everywhere, we denote by $a = g(S\xi,\xi)$ the corresponding principal curvature function on $M$, so $S\xi = a\xi$.

The following result relates some geometric objects on $M$ when the Reeb flow is a geodesic flow.

\begin{prop}\label{GRF} 
Let $M$ be a connected orientable real hypersurface with geodesic Reeb flow in a K\"{a}hler manifold $\bar{M}$. Then
\begin{align*}
& a g((S\phi + \phi S)X,Y) - 2g(S\phi SX,Y) \\
& =  g(\bar{R}_{J\xi}X,JY) - g(\bar{R}_{J\xi}Y,JX) + g(\bar{R}_{J\xi}\xi,JX)\eta(Y) - g(\bar{R}_{J\xi}\xi,JY)\eta(X) .
\end{align*}
\end{prop}

\begin{proof}
Using the Codazzi equation we get
\begin{align*}
g(\bar{R}(X,Y)\xi,J\xi)
& = g((\nabla_XS)Y - (\nabla_YS)X, \xi) 
= g((\nabla_XS)\xi,Y) - g((\nabla_YS)\xi,X) \\
&  = da(X)\eta(Y) - da(Y)\eta(X) + a g((S\phi + \phi S)X,Y) - 2g(S\phi SX,Y).
\end{align*}
Inserting $Y = \xi$ yields
\[
da(X) = da(\xi)\eta(X) + g(\bar{R}(X,\xi)\xi,J\xi) = da(\xi)\eta(X) - g(\bar{R}_{J\xi}\xi,JX),
\] 
and inserting this equation and the corresponding one for $da(Y)$ into the previous equation gives
\begin{align*}
& a g((S\phi + \phi S)X,Y) - 2g(S\phi SX,Y) \\
& =  g(\bar{R}(X,Y)\xi,J\xi) + g(\bar{R}_{J\xi}\xi,JX)\eta(Y) - g(\bar{R}_{J\xi}\xi,JY)\eta(X) .
\end{align*}
Finally, using the algebraic Bianchi identity, we get
\begin{align*}
g(\bar{R}(X,Y)\xi,J\xi)
& = - g(\bar{R}(Y,\xi)X,J\xi) - g(\bar{R}(\xi,X)Y,J\xi) \\
& = - g(\bar{R}(JY,J\xi)X,J\xi) - g(\bar{R}(J\xi,JX)Y,J\xi) \\
& =  g(\bar{R}(JY,J\xi)J\xi,X) - g(\bar{R}(J\xi,JX)Y,J\xi) \\
& = g(\bar{R}_{J\xi}X,JY) - g(\bar{R}_{J\xi}Y,JX).
\end{align*}
This finishes the proof.
\end{proof}

Using the fact that for Hermitian symmetric spaces the Riemannian curvature tensor is parallel, we obtain an alternative equation involving $S\phi + \phi S$.

\begin{prop}\label{GRF2} 
Let $M$ be a connected orientable real hypersurface with geodesic Reeb flow in an Hermitian symmetric space $\bar{M}$. Then
\begin{align*}
& da(\xi)g((S\phi + \phi S)X,Y) \\
& \quad = \eta(X)g(\bar{R}_{J\xi} \xi ,SY)   - \eta(Y)g(\bar{R}_{J\xi} \xi ,SX)   
 - a \eta(X) g(\bar{R}_{J\xi} \xi ,Y) + a \eta(Y)  g(\bar{R}_{J\xi} \xi ,X) \\
& \qquad + 3g(\bar{R}_{J\xi} JY ,JSX) - 3g(\bar{R}_{J\xi} JX,JSY)    + g(\bar{R}_{J\xi} Y,SX) 
  - g(\bar{R}_{J\xi}X,SY) .
\end{align*}
\end{prop}

\begin{proof}
From
\[
da(X) = da(\xi)\eta(X) - g(\bar{R}_{J\xi}\xi,JX) = da(\xi)\eta(X) - g(\bar{R}(\xi,J\xi)J\xi,JX)
\]
we get
\[
\grad^Ma = da(\xi)\xi + J\bar{R}_{J\xi}\xi - g(\bar{R}_{J\xi}\xi,\xi)J\xi
\]
and hence
\[
\hess^a(X,Y)   = g(\nabla_X\grad^Ma,Y)  = Xda(Y) - da(\nabla_XY).
\]
As the Hessian of a function is symmetric, we have
\[
0 = \hess^a(X,Y) - \hess^a(Y,X) = Xda(Y)  - Yda(X) - da([X,Y]) .
\]
Since $\bar{M}$ is an Hermitian symmetric space, we have $\bar\nabla \bar{R} = 0$ and $\bar\nabla J = 0$. Using these identities we first calculate $Xda(Y)$ by 
\begin{align*}
Xda(Y) & = X(da(\xi)\eta(Y) - g(\bar{R}(\xi,J\xi)J\xi,JY)) \\
& = (Xda(\xi))\eta(Y) + da(\xi)(X\eta(Y)) - Xg(\bar{R}(\xi,J\xi)J\xi,JY) \\
& = \hess^a(X,\xi)\eta(Y) + da(\nabla_X\xi)\eta(Y) + da(\xi)\eta(\nabla_XY) + da(\xi)g(Y,\nabla_X\xi) \\
& \qquad - g(\bar\nabla_X\bar{R}(\xi,J\xi)J\xi,JY) - g(\bar{R}(\xi,J\xi)J\xi,\bar\nabla_XJY) \\
& = \hess^a(X,\xi)\eta(Y) + da(\phi SX)\eta(Y) + da(\xi)\eta(\nabla_XY) + da(\xi)g(Y,\phi SX) \\
& \qquad - g(\bar{R}(\bar\nabla_X\xi,J\xi)J\xi,JY) - g(\bar{R}(\xi,\bar\nabla_XJ\xi)J\xi,JY) - g(\bar{R}(\xi,J\xi)\bar\nabla_XJ\xi,JY) \\ 
& \qquad - g(\bar{R}(\xi,J\xi)J\xi,\bar\nabla_XJY) \\
& = \hess^a(X,\xi)\eta(Y) + da(\phi SX)\eta(Y) + da(\xi)\eta(\nabla_XY) + da(\xi)g(Y,\phi SX) \\
& \qquad - 2g(\bar{R}(\bar\nabla_X\xi,J\xi)J\xi,JY) - g(\bar{R}(\xi,J\xi)\bar\nabla_X\xi,Y) - g(\bar{R}(\xi,J\xi)J\xi,J\bar\nabla_XY) \\
& = \hess^a(X,\xi)\eta(Y) + da(\phi SX)\eta(Y) + da(\xi)\eta(\nabla_XY) + da(\xi)g(Y,\phi SX) \\
& \qquad - 2g(\bar{R}(JSX,J\xi)J\xi,JY)   - g(\bar{R}(\xi,J\xi)JSX,Y) - g(\bar{R}(\xi,J\xi)J\xi,J\bar\nabla_XY).
\end{align*}
This implies
\begin{align*}
0 & =  Xda(Y)  - Yda(X) - da([X,Y]) \\
& = \hess^a(X,\xi)\eta(Y) + da(\phi SX)\eta(Y) + da(\xi)\eta(\nabla_XY) + da(\xi)g(Y,\phi SX) \\
& \qquad - 2g(\bar{R}(JSX,J\xi)J\xi,JY)   - g(\bar{R}(\xi,J\xi)JSX,Y)  - g(\bar{R}(\xi,J\xi)J\xi,J\bar\nabla_XY) \\
& \qquad - \hess^a(Y,\xi)\eta(X) - da(\phi SY)\eta(X) - da(\xi)\eta(\nabla_YX) - da(\xi)g(X,\phi SY) \\
& \qquad + 2g(\bar{R}(\phi SY,J\xi)J\xi,JX)   + g(\bar{R}(\xi,J\xi)J SY,X) + g(\bar{R}(\xi,J\xi)J\xi,J\bar\nabla_YX) \\
& \qquad - da(\xi)\eta([X,Y]) + g(\bar{R}(\xi,J\xi)J\xi,J[X,Y]) \displaybreak[0] \\ 
& = \hess^a(X,\xi)\eta(Y) + da(\phi SX)\eta(Y)  + da(\xi)g(Y,\phi SX) \\
& \qquad - 2g(\bar{R}(JSX,J\xi)J\xi,JY)   - g(\bar{R}(\xi,J\xi)JSX,Y) \\ 
& \qquad - \hess^a(Y,\xi)\eta(X) - da(\phi SY)\eta(X)  - da(\xi)g(X,\phi SY) \\
& \qquad + 2g(\bar{R}(JSY,J\xi)J\xi,JX)   + g(\bar{R}(\xi,J\xi)J SY,X)  \\
& = \hess^a(X,\xi)\eta(Y) - \hess^a(Y,\xi)\eta(X) + da(\phi SX)\eta(Y)  - da(\phi SY)\eta(X) \\
& \qquad + da(\xi)g((S\phi + \phi S)X,Y) \\
& \qquad - 2g(\bar{R}(JSX,J\xi)J\xi,JY)   - g(\bar{R}(\xi,J\xi)JSX,Y)  \\ 
& \qquad + 2g(\bar{R}(JSY,J\xi)J\xi,JX)   + g(\bar{R}(\xi,J\xi)J SY,X) .
\end{align*}
Inserting $Y = \xi$ gives
\[
 \hess^a(X,\xi) = \hess^a(\xi,\xi)\eta(X) - da(\phi SX)  + g(\bar{R}(\xi,J\xi)J\xi,SX)    - a g(\bar{R}(\xi,J\xi)J \xi,X) .
\]
Inserting this and the analogous equation for $\hess^a(Y,\xi)$ into the previous equation leads to
\begin{align*}
da(\xi)g((S\phi + \phi S)X,Y) & =   \eta(X)g(\bar{R}(\xi,J\xi)J\xi,SY) - \eta(Y)g(\bar{R}(\xi,J\xi)J\xi,SX) \\
& \qquad  - a \eta(X)  g(\bar{R}(\xi,J\xi)J \xi,Y) + a \eta(Y)  g(\bar{R}(\xi,J\xi)J \xi,X)  \\
& \qquad + 2g(\bar{R}(JSX,J\xi)J\xi,JY)   + g(\bar{R}(\xi,J\xi)JSX,Y)  \\ 
& \qquad - 2g(\bar{R}(JSY,J\xi)J\xi,JX)   - g(\bar{R}(\xi,J\xi)J SY,X)  .
\end{align*}
Using again the algebraic Bianchi identity, we can rewrite
\[
g(\bar{R}(\xi,J\xi)JSX,Y) = g(\bar{R}(JSX,J\xi)J\xi,JY) + g(\bar{R}(Y,J\xi)J\xi,SX) .
\]
Inserting this equation, and the corresponding one with $X$ and $Y$ interchanged, into the previous equation leads to the equation in the assertion.
\end{proof}

As a consequence we derive some useful equations involving the principal curvatures of $M$.

\begin{cor}\label{GRF3} 
Let $M$ be a connected orientable real hypersurface with geodesic Reeb flow in an Hermitian symmetric space $\bar{M}$. Let $X,Y \in {\mathcal C}$ with $SX = b X$ and $SY = c Y$. Then we have
\[
(a (b + c) - 2bc)g(JX,Y)  = g(\bar{R}_{J\xi}X,JY) - g(\bar{R}_{J\xi}Y,JX)
\]
and
\[
(b+c) da(\xi)g(JX,Y) = 3(b-c) g(\bar{R}_{J\xi}JX,JY)   + (b-c) g(\bar{R}_{J\xi}X,Y)  .
\]
\end{cor}

\begin{proof}
The equations follow immediately when inserting $X,Y \in {\mathcal C}$ with $SX = b X$ and $SY = c Y$ into Proposition \ref{GRF} and Proposition \ref{GRF2}, respectively.
\end{proof}

We now turn our attention to the situation when the Reeb flow is an isometric flow and start with a useful characterization.

\begin{prop} \label{isomcomm}
Let $M$ be a connected orientable real hypersurface in a K\"{a}hler manifold $\bar{M}$. Then the Reeb flow on $M$ is an isometric flow if and only if $S\phi = \phi S$.
\end{prop}

\begin{proof}
The Reeb flow on $M$ is an isometric flow if and only if $\xi$ is a Killing vector field. This is equivalent for $\nabla \xi$ to be a skew-symmetric tensor on $M$, which means
\begin{align*}
0 & = g(\nabla_X \xi, Y) + g(X,\nabla_Y\xi) = -g(\bar\nabla_X J\zeta,Y) - g(X,\bar\nabla_Y J\zeta) \\
& = -g(J\bar\nabla_X\zeta,Y) - g(X,J\bar\nabla_Y\zeta) = g(JSX,Y) - g(X,JSY) \\
& = g(\phi SX,Y) - g(X,\phi SY) = g((\phi S - S\phi)X,Y)
\end{align*}
for all vector fields $X,Y$ on $M$. This implies the assertion.
\end{proof}

If $S\phi = \phi S$ holds, then $0 = S\phi\xi = \phi S\xi$. From Proposition \ref{geodReeb} and Proposition \ref{isomcomm} we therefore get:

\begin{cor}\label{isomgeod}
Let $M$ be a connected orientable real hypersurface in a K\"{a}hler manifold $\bar{M}$. If the Reeb flow on $M$ is an isometric flow, then it is also a geodesic flow.
\end{cor}

The equation $S\phi = \phi S$ leads to further geometric properties of $M$. First of all, by differentiating $S\phi = \phi S$ we obtain the following result relating the shape operator $S$, the structure tensor $\phi$ and the normal Jacobi operator $\bar{R}_{J\xi}$.

\begin{prop}\label{IRF}
Let $M$ be a connected orientable real hypersurface with isometric Reeb flow in a K\"{a}hler manifold $\bar{M}$. Then
\[
  (S-a I)S\phi X = \bar{R}_{J\xi}J X - g(\bar{R}_{J\xi}J X,\xi)\xi = ( \bar{R}_{J\xi}JX  )_{\mathcal C} .
\]
\end{prop}

\begin{proof}
Since $M$ has isometric Reeb flow, we have $S\phi = \phi S$ by Proposition \ref{isomcomm}. 
Differentiating $S\phi = \phi S$ covariantly we obtain
\[
(\nabla_XS)\phi Y + S(\nabla_X\phi)Y = (\nabla_X\phi)SY + \phi(\nabla_XS)Y.
\]
The tangential part of the K\"{a}hler condition $(\bar\nabla_XJ)Y = 0$ on $\bar{M}$ gives
\[
(\nabla_X\phi)Y = \eta(Y)SX - g(SX,Y)\xi.
\]
Inserting this into the previous equation leads to
\[
(\nabla_XS)\phi Y - \phi(\nabla_XS)Y = a g(SX,Y)\xi - g(SX,SY)\xi  + a\eta(Y)SX - \eta(Y)S^2X .
\]
Taking inner product with a vector field $Z$ tangent to $M$ gives
\begin{align*}
& g((\nabla_XS)Y,\phi Z) + g((\nabla_XS)Z,\phi Y) \\
& =  a\eta(Z)g(SX,Y) - \eta(Z)g(SX,SY) + a\eta(Y)g(SX,Z)  - \eta(Y)g(SX,SZ) ,
\end{align*}
which implies
\begin{align*}
& g((\nabla_XS)Y,\phi Z) + g((\nabla_XS)Z,\phi Y) + g((\nabla_YS)Z,\phi X) + g((\nabla_YS)X,\phi Z) \\
& \qquad - g((\nabla_ZS)X,\phi Y) - g((\nabla_ZS)Y,\phi X) \\
& = 2a\eta(Z)g(SX,Y) - 2\eta(Z)g(SX,SY) .
\end{align*}
The left-hand side of this equation can be rewritten as
\begin{align*}
& 2g((\nabla_XS)Y,\phi Z) - g((\nabla_XS)Y - (\nabla_YS)X,\phi Z) \\
& + g((\nabla_YS)Z - (\nabla_ZS)Y,\phi X) - g((\nabla_ZS)X - (\nabla_XS)Z,\phi Y), 
\end{align*}
and using the Codazzi equation this implies
\begin{align*}
2g((\nabla_X S)Y,\phi Z)
& = 2a\eta(Z)g(SX,Y) - 2\eta(Z)g(SX,SY) \\
& \qquad + g(\bar{R}(X,Y)\phi Z,J\xi) - g(\bar{R}(Y,Z)\phi X,J\xi) + g(\bar{R}(Z,X)\phi Y,J\xi).
\end{align*}
Replacing $Z$ by $\phi Z$, the left-hand side of the previous equation becomes
\[
- 2g((\nabla_XS)Y,Z) + 2\eta(Z)g((\nabla_XS)Y, \xi).
\]
Since
\[ g((\nabla_XS)Y, \xi) = g((\nabla_XS)\xi, Y) =
g(\nabla_X(a\xi),Y) - g(S\nabla_X\xi,Y),
\]
we see that
\[
2da(X)\eta(Y)\eta(Z) + 2a\eta(Z)g(\phi SX,Y) - 2\eta(Z) g(S\phi SX,Y) - 2g((\nabla_XS)Y,Z)
\]
becomes the left-hand side of the previous equation when we replace $Z$ by $\phi Z$.
Replacing $Z$ by $\phi Z$ also on the right-hand side we get
\begin{align*}
& 2da(X)\eta(Y)\eta(Z) + 2a\eta(Z)g(\phi SX,Y) - 2\eta(Z) g(S\phi SX,Y) - 2g((\nabla_XS)Y,Z) \\
& \quad = g(\bar{R}(X,Y)\phi^2 Z,J\xi) - g(\bar{R}(Y,\phi Z)\phi X,J\xi) + g(\bar{R}(\phi Z,X)\phi Y,J\xi) \\
& \quad = - g(\bar{R}(X,Y)J\xi,\phi^2 Z) - g(\bar{R}(\phi X,J\xi)Y,\phi Z) - g(\bar{R}(\phi Y,J\xi)X,\phi Z) \\
& \quad =  g(\bar{R}(X,Y)J\xi,Z) - \eta(Z)g(\bar{R}(X,Y)J\xi,\xi) \\
& \qquad \quad - g(\bar{R}(\phi X,J\xi)Y,\phi Z) - g(\bar{R}(\phi Y,J\xi)X,\phi Z),
\end{align*}
or equivalently,
\begin{align*}
2g((\nabla_XS)Y,Z) 
& = 2da(X)\eta(Y)\eta(Z) + 2a\eta(Z)g(\phi SX,Y) - 2\eta(Z) g(S\phi SX,Y) \\
& \quad - g(\bar{R}(X,Y)J\xi,Z) + \eta(Z)g(\bar{R}(X,Y)J\xi,\xi) \\
& \quad + g(\bar{R}(\phi X,J\xi)Y,\phi Z) + g(\bar{R}(\phi Y,J\xi)X,\phi Z).
\end{align*}
Since $JZ = \phi Z + \eta(Z)J\xi$, this implies
\begin{align*}
2g((\nabla_XS)Y,Z) 
& = \left\{ 2da(X)\eta(Y) + 2a g(\phi SX,Y) - 2 g(S\phi SX,Y) \right. \\
& \qquad \left.  - g(\bar{R}(X,Y)\xi,J\xi) - g(\bar{R}(\phi X,J\xi)Y,J\xi) - g(\bar{R}(\phi Y,J\xi)X,J\xi) \right\} \eta(Z) \\
& \quad - g(\bar{R}(X,Y)J\xi,Z) - g(J\bar{R}(\phi X,J\xi)Y,Z) - g(J\bar{R}(\phi Y,J\xi)X,Z) .
\end{align*}
Since $\bar{M}$ is a K\"{a}hler manifold, we have
\begin{align*}
g(J\bar{R}(\phi X,J\xi)Y,Z) & = g(\bar{R}(\phi X,J\xi)JY,Z) \\ 
& = g(\bar{R}(\phi X,J\xi)\phi Y,Z) + \eta(Y)g(\bar{R}(\phi X,J\xi)J\xi,Z).
\end{align*}
The previous equation then becomes
\begin{align*}
2g((\nabla_XS)Y,Z) 
& = \left\{ 2da(X)\eta(Y) + 2a g(\phi SX,Y) - 2 g(S\phi SX,Y) \right. \\
& \qquad \left.  - g(\bar{R}(X,Y)\xi,J\xi) - g(\bar{R}(\phi X,J\xi)Y,J\xi) - g(\bar{R}(\phi Y,J\xi)X,J\xi) \right\} \eta(Z)  \\
& \qquad - \eta(Y)g(\bar{R}(\phi X,J\xi)J\xi,Z) - \eta(X)g(\bar{R}(\phi Y,J\xi)J\xi,Z) \\
& \quad - g(\bar{R}(X,Y)J\xi,Z) - g(\bar{R}(\phi X,J\xi)\phi Y,Z) - g(\bar{R}(\phi Y,J\xi)\phi X,Z).
\end{align*}
Inserting $Y = \xi$ and using basic curvature identities for K\"{a}hler manifolds, we obtain
\begin{align*}
2g((\nabla_XS)\xi,Z) 
& = \left\{ 2da(X)  - g(\bar{R}(X,\xi)\xi,J\xi) + g(\bar{R}(\phi X,J\xi)J\xi,\xi) \right\} \eta(Z) \\
& \qquad - g(\bar{R}(\phi X,J\xi)J\xi,Z) - g(\bar{R}(X,\xi)J\xi,Z)  \\
& = \left\{ 2da(X)  + g(\bar{R}(JX,J\xi)J\xi,\xi) + g(\bar{R}(\phi X,J\xi)J\xi,\xi) \right\} \eta(Z) \\
& \qquad - g(\bar{R}(\phi X,J\xi)J\xi,Z) - g(\bar{R}(JX,J\xi)J\xi,Z)  \\
& = \left\{ 2da(X)  + 2g(\bar{R}(\phi X,J\xi)J\xi,\xi) \right\} \eta(Z) - 2g(\bar{R}(\phi X,J\xi)J\xi,Z) .
\end{align*}
On the other hand, we have
\[
(\nabla_XS)\xi
= da(X)\xi + a\nabla_X\xi - S\nabla_X\xi
= da(X)\xi + a S \phi X - S^2 \phi X.
\]
Inserting this into the previous equation implies
\[
a g( S \phi X,Z)  - g(S^2 \phi X,Z)  = g(\bar{R}(\phi X,J\xi)J\xi,\xi)  \eta(Z) - g(\bar{R}(\phi X,J\xi)J\xi,Z),
\]
which implies the assertion.
\end{proof}

The next result provides useful information on the eigenspaces of the normal Jacobi operator if the Reeb flow is an isometric flow.

\begin{prop}\label{IRF2}
Let $M$ be a connected orientable real hypersurface with isometric Reeb flow in a K\"{a}hler manifold $\bar{M}$. Then we have 
\[
 g(\bar{R}_{J\xi}X,Y) = g(\bar{R}_{J\xi}JX,JY) 
\]
for all $X,Y \in {\mathcal C}$.
\end{prop}

\begin{proof}
For $X \in {\mathcal C}$ we have $\phi X = JX$. From Proposition \ref{GRF} we get
\[
2a g(S J X,J X) - 2g(S^2J X,J X) 
=  - g(\bar{R}_{J\xi}X,X) - g(\bar{R}_{J\xi}JX,JX)  ,
\]
and from Proposition \ref{IRF} we get
\[
2a g(SJX,J X) - 2g(S^2JX,JX) = 
- 2g((S-a I)SJX,JX) = - 2g(\bar{R}_{J\xi}J X,JX) .
\]
Comparing both equations leads to $g(\bar{R}_{J\xi}X,X) = g(\bar{R}_{J\xi}JX,JX)$ for all $X \in {\mathcal C}$.
For $X,Y \in {\mathcal C}$ we therefore get
\begin{align*}
& 2g(\bar{R}_{J\xi}X,Y) + g(\bar{R}_{J\xi}X,X) + g(\bar{R}_{J\xi}Y,Y) \\
& = g(\bar{R}_{J\xi}(X+Y),X+Y) = g(\bar{R}_{J\xi}J(X+Y),J(X+Y)) \\
& = g(\bar{R}_{J\xi}(JX+JY),JX+JY) \\
& = 2g(\bar{R}_{J\xi}JX,JY) + g(\bar{R}_{J\xi}JX,JX) + g(\bar{R}_{J\xi}JY,JY)
\end{align*}
Since $g(\bar{R}_{J\xi}X,X) = g(\bar{R}_{J\xi}JX,JX)$ and $g(\bar{R}_{J\xi}Y,Y) = g(\bar{R}_{J\xi}JY,JY) $, the assertion follows.
\end{proof}

From Proposition \ref{IRF2} and Corollary \ref{GRF3} we obtain more information on the principal curvatures.

\begin{cor}\label{IRF3} 
Let $M$ be a connected orientable real hypersurface with isometric Reeb flow in an Hermitian symmetric space $\bar{M}$. Let $X,Y \in {\mathcal C}$ with $SX = b X$ and $SY = c Y$. Then we have
\[
(a (b + c) - 2bc)g(JX,Y)  = 2g(\bar{R}_{J\xi}X,JY) 
\]
and
\[
(b+c) da(\xi)g(JX,Y) = 4(b-c) g(\bar{R}_{J\xi}X,Y)  .
\]
\end{cor}

The next result provides a characterization for the constancy of the principal curvature function $a$ in terms of a simple condition on the normal Jacobi operator. 

\begin{prop} \label{IRF4}
Let $M$ be a connected orientable real hypersurface with isometric Reeb flow in an irreducible Hermitian symmetric space $\bar{M}$. Then $a$ is constant if and only if $\xi$ is an eigenvector of $\bar{R}_{J\xi}$.
\end{prop}

\begin{proof}
Recall that $S\phi = \phi S$ when $M$ has isometric Reeb flow. In Corollary \ref{IRF3} we can thus choose $Y = JX$ and $c = b$. If also $\|X \| = 1$, we get $b^2  - a b  -g(\bar{R}_{J\xi}X,X) = 0$ and $b da(\xi) = 0 $.

We first assume that $da(\xi) \neq 0$ at some point $p \in M$. Then, by continuity, $da(\xi) \neq 0$ on an open neighborhood $U$ of $p$ in $M$. The following calculations are valid on $U$. From $b da(\xi) = 0$ we see that $SX = 0$ for all $X \in {\mathcal C}$. From Corollary \ref{IRF3} we then get $\bar{R}_{J\xi}({\mathcal C}) = \RR \xi$. Now consider the Codazzi equation $g(\bar{R}(X,Y)Z,J\xi) = g((\nabla_XS)Y,Z) - g((\nabla_YS)X,Z)$. For $X,Y,Z \in {\mathcal C}$ we get $g((\nabla_XS)Y,Z) = g(\nabla_XSY,Z) - g(S\nabla_XY,Z) = 0$
and hence $g(\bar{R}(X,Y)Z,J\xi) = 0$. We also get $g(\bar{R}(X,Y)Z,\xi) = g(\bar{R}(X,Y)JZ,J\xi) = g((\nabla_XS)Y,JZ) - g((\nabla_YS)X,JZ) = 0$. Altogether this implies that $\bar{R}({\mathcal C},{\mathcal C}){\mathcal C} \subset {\mathcal C}$. In other words, ${\mathcal C}$ is curvature-invariant at each point in $U$. Thus there exists a totally geodesic submanifold $\Sigma$ of $\bar{M}$ with $T_p\Sigma = {\mathcal C}_p$ (see e.g.\ Theorem 10.3.3 in \cite{BCO16}). Since ${\mathcal C}$ is $J$-invariant, $\Sigma$ is a totally geodesic complex hypersurface in $\bar{M}$. It follows that the index of $\bar{M}$ is at most $2$ (see \cite{BO17}). The only irreducible Hermitian symmetric spaces with index $\leq 2$ are complex projective spaces, complex quadrics and their noncompact dual symmetric spaces (see \cite{BO17}), which all admit totally geodesic complex hypersurfaces. In all these cases it is known that $a$ is constant (see \cite{BS13, MR86, Ok75, Su17}). The assumption $da(\xi) \neq 0$ therefore leads to a contradiction. Thus we must have $da(\xi) = 0$. 

Since $da(X) = da(\xi)\eta(X) - g(\bar{R}_{J\xi}\xi,JX)$ for all $X \in TM$ (see proof of Proposition \ref{GRF}), this implies $da(X) = - g(\bar{R}_{J\xi}\xi,JX)$ for all $X \in TM$. This finishes the proof.
\end{proof}

We now use the previous characterization to show that $a$ is constant when $M$ has isometric Reeb flow. 

\begin{prop} \label{IRF5}
Let $M$ be a connected orientable real hypersurface with isometric Reeb flow in an irreducible Hermitian symmetric space $\bar{M}$. Then $a$ is constant.
\end{prop}

\begin{proof}
We saw in the proof of Proposition \ref{IRF4} that $da(\xi) = 0$. Inserting $X = \xi$ into the equation in Proposition \ref{GRF2} implies
$0 =   a  (g(\bar{R}_{J\xi}\xi,Y) - \eta(Y)  g(\bar{R}_{J\xi}\xi,\xi) )$.
If $a$ is nonzero at some point $p \in M$, then, since $a$ is smooth, $a$ is nonzero in some open neighborhood $U$ of $p \in M$. The previous equation then yields 
$g(\bar{R}_{J\xi}\xi,Y) = \eta(Y)  g(\bar{R}_{J\xi}\xi,\xi)$
on $U$ and hence $\bar{R}_{J\xi}\xi \in \RR\xi$ on $U$. Proposition \ref{IRF4} then implies that $a$ is constant on $U$. So, whenever $a$ is nonzero at some point, it is constant in an open neighborhood of that point. Since $M$ is connected and $a$ is a smooth function, it follows that $a$ is constant on $M$. 
\end{proof}

From Proposition \ref{IRF5} and Corollary \ref{IRF3} we obtain more information on the principal curvatures.

\begin{cor}\label{IRF6} 
Let $M$ be a connected orientable real hypersurface with isometric Reeb flow in an irreducible Hermitian symmetric space $\bar{M}$. Let $X,Y \in {\mathcal C}$ with $SX = b X$ and $SY = c Y$. Then we have
\[
(a (b + c) - 2bc)g(JX,Y)  = 2g(\bar{R}_{J\xi}X,JY) 
\mbox{ and }
0 = (b-c) g(\bar{R}_{J\xi}X,Y)  .
\]
\end{cor}

From the second of the two equations in Corollary \ref{IRF6} we see that the normal Jacobi operator $\bar{R}_{J\xi}$ leaves the principal curvature spaces of $M$ invariant. This means that $S$ and $\bar{R}_{J\xi}$ can be simultaneously diagonalized, or equivalently, $\bar{R}_{J\xi} S = S\bar{R}_{J\xi}$. Submanifolds with such a property are called curvature-adapted (see \cite{BV92}) or compatible (see \cite{Gr04}). We thus have:

\begin{cor} \label{IRF7}
Let $M$ be a connected orientable real hypersurface with isometric Reeb flow in an irreducible Hermitian symmetric space $\bar{M}$. Then $M$ is curvature-adapted.
\end{cor}

From the first of the two equations in Corollary \ref{IRF6} we get additional information by choosing $Y = \phi X = JX$ and $\|X\| = 1$, namely $b^2 -  a b - \kappa = 0$,
where $\kappa \in \RR$ with $\bar{R}_{J\xi}X = \kappa X$. It follows that the principal curvatures of $M$ are completely determined by $a$ and the eigenvalues of the normal Jacobi operator. 

Recall that a tangent vector $X \in T_p\bar{M}$ of a semisimple Riemannian symmetric space $\bar{M}$  is said to be regular if there exists a unique connected, complete, totally geodesic flat submanifold $F$ of $\bar{M}$ (a so-called maximal flat in $\bar{M}$) with $p \in F$, $X\in T_pF$ and $\dim(F) = \rk(\bar{M})$. Otherwise $X$ is said to be a singular tangent vector of $\bar{M}$. It is known from standard theory of symmetric spaces that every tangent vector is tangent to some maximal flat. The next result shows that, if the Reeb flow is isometric, then the normal vectors of $M$ are singular tangent vectors under some mild additional assumptions. 

\begin{prop} \label{normalsingular}
Let $M$ be a connected orientable real hypersurface with isometric Reeb flow in an irreducible Hermitian symmetric space $\bar{M}$ with $\rk(\bar{M}) \geq 2$. Then $J\xi$ is a singular tangent vector of $\bar{M}$ everywhere.
\end{prop}

\begin{proof}
We prove this by contradiction. Assume that $J\xi_p$ is a regular tangent vector of $\bar{M}$ at some point $p \in M$. Then there exists a unique maximal flat $F$ of $\bar{M}$ with $p \in F$ and $J\xi_p \in T_pF$. It is known that any maximal flat of an irreducible Hermitian symmetric space is a totally real submanifold. Since $\dim(F) = \rk(\bar{M}) \geq 2$, there exists a unit vector $X \in T_pF$ perpendicular to $J\xi_p$. Note that $X \in {\mathcal C}_p$ since $F$ is totally real. Since $F$ is totally geodesic in $\bar{M}$, the Gau{\ss} equation for $F$ in $\bar{M}$ implies that $\bar{R}_{J\xi_p}X = 0$. As $M$ has isometric Reeb flow, we get $\bar{R}_{J\xi_p}JX = 0$ from Proposition \ref{IRF2}. In particular, the sectional curvature $\bar{K}(V)$ of $\bar{M}$ with respect to the $2$-plane $V = \RR J\xi_p \oplus \RR JX \subset T_p\bar{M}$ satisfies $\bar{K}(V) = 0$. This contradicts the assumption that $J\xi_p$ is a regular tangent vector of $\bar{M}$, because the only $2$-planes $V \subset T_p\bar{M}$ with $J\xi_p \in V$ and $\bar{K}(V) = 0$ are those contained in $T_pF$.
\end{proof}

The next result shows that the normal spaces of a real hypersurface with isometric Reeb flow generate spaces of constant curvature.

\begin{prop} \label{normalline}
Let $M$ be a connected orientable real hypersurface with isometric Reeb flow in an irreducible Hermitian symmetric space $\bar{M}$ of compact type (resp.\ of noncompact type). For each $p \in M$ there exists a totally geodesic complex projective (resp.\ hyperbolic) line $\Sigma = \CC P^1$ (resp.\ $\Sigma = \CC H^1$) in $\bar{M}$ with $p \in \Sigma$ such that $T_p\Sigma = \CC\xi_p = \RR \xi_p \oplus \RR J\xi_p$. \end{prop}

\begin{proof}
Let $p \in M$ and $V = \CC \xi_p$. From Propositions  \ref{IRF4} and \ref{IRF5} we know that $\bar{R}_{J\xi} \xi = \kappa \xi$. Using curvature identities this implies
$\kappa g(\xi,X) = g(\bar{R}(\xi,J\xi)J\xi,X) = g(\bar{R}(J\xi,\xi)\xi,JX) = -g(J\bar{R}(J\xi,\xi)\xi,X)$.
Thus $- J\bar{R}(J\xi,\xi)\xi = \kappa\xi$ and hence $\bar{R}_\xi J\xi = \kappa J \xi$. It follows that
$V$ is a curvature-invariant subspace of $T_p\bar{M}$. Thus there exists a connected, complete, totally geodesic submanifold $\Sigma$ of $\bar{M}$ with $T_p\Sigma = V$. Since $V$ is $J$-invariant, the submanifold $\Sigma$ is a complex submanifold of $\bar{M}$. Moreover, from $\bar{R}_{J\xi_p} \xi_p = \kappa \xi_p$ we see that $\Sigma$ has constant sectional curvature $\kappa$. 

First assume that $\kappa = 0$. Then $\Sigma$ is a flat totally geodesic submanifold of $\bar{M}$ and must be contained in a maximal flat of $\bar{M}$. Maximal flats in irreducible Hermitian symmetric spaces are totally real submanifolds. Hence $\Sigma$ is totally real, which is a contradiction. It follows that $\kappa \neq 0$. 

If $\bar{M}$ is of compact type, then $\Sigma$ is a complex projective line $\CC P^1$ with constant sectional curvature $\kappa > 0$. If $\bar{M}$ is of noncompact type, then $\Sigma$ is a complex hyperbolic line $\CC H^1$ with constant sectional curvature $\kappa < 0$.  
\end{proof}

Since every geodesic in complex projective line is closed, we get the following consequence from Corollary \ref{isomgeod} and Proposition \ref{normalline}.

\begin{cor}
Let $M$ be a connected orientable real hypersurface with isometric Reeb flow in an irreducible Hermitian symmetric space $\bar{M}$ of compact type. Then the Reeb flow lines are closed geodesics in $M$ and the geodesics in $\bar{M}$ that are perpendicular to $M$ are closed.
\end{cor}

\begin{proof}
The second statement follows immediately from Corollary \ref{isomgeod} and Proposition \ref{normalline}. For the first statement, we use the Weingarten formula and Corollary \ref{isomgeod} to obtain $\bar\nabla_\xi\xi = \nabla_\xi\xi + a J\xi = a J\xi$.
Differentiating again and using Proposition \ref{IRF5} we get $\bar\nabla_\xi\bar\nabla_\xi\xi = a \bar\nabla_\xi J\xi = a J \bar\nabla_\xi \xi = -a^2 \xi$. Using Proposition \ref{normalline}, we conclude that each integral curve of $\xi$ is a ``small'' circle in a complex projective line, which is a sphere of positive constant sectional curvature, and hence a closed curve. 
\end{proof}

This finishes the general structure theory for real hypersurfaces with isometric Reeb flow in K\"{a}hler manifolds. In the next sections we will apply this structure theory to investigate real hypersurfaces with isometric Reeb flow in some Hermitian symmetric spaces, with the aim to obtain a classification of such hypersurfaces.

\section{Structure theory of compact Hermitian symmetric spaces} \label{StcHss}

In this section we present an algebraic model of Hermitian symmetric spaces of compact type based on structure theory of semisimple real and complex Lie algebras. Some algebraic details can be found for example in the book \cite{Sa90} by Samelson.

Let $\bar{M}$ be an irreducible Hermitian symmetric space of compact type and $G = I^o(\bar{M})$ be the identity component of the isometry group of $\bar{M}$. We choose and fix a point $o \in \bar{M}$ and denote by $K$ the isotropy group of $G$ at $o$. Then $\bar{M}$ can be realized as the homogeneous space $\bar{M} = G/K$ in the usual way. We denote by $\cg$ and $\ck$ the Lie algebras of $G$ and $K$, respectively. Then $\cg$ is a simple real Lie algebra.

The isotropy group $K$ has a $1$-dimensional center $Z$ and coincides with the centralizer of $Z$ in $G$, which implies that $K$ has maximal rank in $G$. There exists a unique element $z_o$ in the Lie algebra $\cz$ of $Z$ such that the complex structure $J$ on $T_o\bar{M}$ is given by $J = \ad(z_o)$. We will give a more explicit description of $z_o$ further below.

Let $\ch$ be a Cartan subalgebra of $\ck$. Since $K$ has maximal rank in $G$, $\ch$ is a Cartan subalgebra of $\cg$. Thus the complexification $\ch^\CC$ of $\ch$ is a Cartan subalgebra of the complexification $\cg^\CC$ of $\cg$. Let
\[
\cg^\CC = \ch^\CC \oplus \left( \bigoplus_{\alpha\in\Delta} \cg_\alpha \right)
\]
be the root decomposition of $\cg^\CC$ with respect to $\ch^\CC$. We fix a set $\{\alpha_1,\ldots,\alpha_r\}$ of simple roots of $\Delta$ and denote by $\Delta^+$ the resulting subset of $\Delta$ consisting of all positive roots.

Let $\cg = \ck \oplus \cp$ be the Cartan decomposition of $\cg$ with respect to $\ck$. Since $[\ch^\CC,\ck^\CC] \subset \ck^\CC$ and $[\ch^\CC,\cp^\CC] \subset \cp^\CC$, we either have $\cg_\alpha \subset \ck^\CC$ or $\cg_\alpha \subset \cp^\CC$ for each $\alpha \in \Delta$. If $\cg_\alpha \subset \ck^\CC$, then the root $\alpha$ is compact, and if $\cg_\alpha \subset \cp^\CC$, then the root $\alpha$ is noncompact. A root $\alpha \in \Delta$ is compact if and only if $\alpha (\cz^\CC) = \{0\}$. Denote by $\Delta_K$ the set of compact roots and by $\Delta_M$ the set of noncompact roots. The decomposition $\cg^\CC = \ck^\CC \oplus \cp^\CC$ is given by
\[
\ck^\CC = \ch^\CC \oplus \left( \bigoplus_{\alpha \in \Delta_K} \cg_\alpha \right)\ ,\ 
\cp^\CC =   \bigoplus_{\alpha \in \Delta_M} \cg_\alpha \ .
\]

We can be more explicit about this. Let $H^1,\ldots,H^r \in \ch^\CC$ be the dual basis of $\alpha_1,\ldots,\alpha_r$ defined by $\alpha_i(H^j) = \delta_{ij}$. There exists an integer $k \in \{1,\ldots,r\}$ such that $\Delta_K$ is generated by the simple roots $\alpha_1,\ldots,\alpha_{k-1},\alpha_{k+1},\ldots,\alpha_r$ and the multiplicity of $\alpha_k$ in the highest root $\delta \in \Delta^+$ is equal to one. The complex structure $J = \ad(z_o)$ on $\cp \cong T_o\bar{M}$ is given by $z_o = iH^k$ and we have
\[
\Delta_K = \{\alpha \in \Delta : \alpha(H^k) = 0\}\ ,\ 
\Delta_M = \{\alpha \in \Delta : \alpha(H^k) = \pm 1\}\ ,
\]
and thus
\[
\ck^\CC = \ch^\CC \oplus \left( \bigoplus_{\substack{\alpha \in \Delta \\ \alpha(H^k) = 0}} \cg_\alpha \right)\ ,\ 
\cp^\CC = \bigoplus_{\substack{\alpha \in \Delta \\ \alpha(H^k) = \pm 1}} \cg_\alpha\ .
\]
Conversely, if there is a simple root whose coefficient in the highest root is equal to one, we can construct a Hermitian symmetric space of compact type from it by applying the Borel-de Siebenthal construction method (\cite{BS49}). 

For each $\alpha \in \Delta$ there exists a unique vector $h_\alpha \in [\cg_\alpha,\cg_{-\alpha}] \subset \ch^\CC$, the so-called coroot corresponding to $\alpha$, such that $\alpha(h_\alpha) = 2$. Then we have
\[
i\ch = \mbox{span\ of}\ \{h_\alpha : \alpha \in \Delta\}
\]
and $\alpha(i\ch) = \RR$ for all $\alpha \in \Delta$. The Killing form $B$ of $\cg^\CC$ is nondegenerate on $\ch^\CC$ and positive definite on $i\ch$. Define $H_\alpha \in \ch^\CC$, the so-called root vector of $\alpha$, by $\alpha(H) = B(H_\alpha,H)$ for all $H \in \ch^\CC$ and a positive definite inner product $(\cdot,\cdot)$ on the dual space $(i\ch)^*$ by linear extension of $(\alpha_\nu,\alpha_\mu) = B(H_{\alpha_\nu},H_{\alpha_\mu})$. Then we have $h_\alpha = \frac{2}{(\alpha,\alpha)}H_\alpha$ for all $\alpha \in \Delta$.

For $\alpha,\beta \in \Delta$ with $\beta \neq \pm \alpha$ the $\alpha$-string containing $\beta$ is the set of roots
\[
\beta -p_{\alpha,\beta}\alpha,\ldots,\beta-\alpha,\beta,\beta+\alpha,\ldots,\beta+q_{\alpha,\beta}\alpha \in \Delta
\]
with $p_{\alpha,\beta},q_{\alpha,\beta} \in \ZZ$ and $p_{\alpha,\beta},q_{\alpha,\beta} \geq 0$ so that $\beta-(p_{\alpha,\beta}+1)\alpha \notin \Delta$ and $\beta + (q_{\alpha,\beta}+1)\alpha \notin \Delta$. The $\alpha$-string containing $\beta$ contains at most four roots. The Cartan integer $c_{\beta,\alpha}$ of $\alpha,\beta \in \Delta$ is defined by 
\[
c_{\beta,\alpha} = \beta(h_\alpha) = B(h_\alpha,H_\beta) = \frac{2}{(\alpha,\alpha)}B(H_\beta,H_\alpha) = 2\frac{(\beta,\alpha)}{(\alpha,\alpha)} \in \{0,\pm 1,\pm 2,\pm 3\}.
\]
The Cartan integer $c_{\beta,\alpha}$ is related to the $\alpha$-string containing $\beta$ by
\[
c_{\beta,\alpha} = p_{\alpha,\beta}-q_{\alpha,\beta}.
\]

For each nonzero $e_\alpha \in \cg_\alpha$ there exists $e_{-\alpha} \in \cg_{-\alpha}$ such that $B(e_\alpha,e_{-\alpha}) = \frac{2}{(\alpha,\alpha)}$. For such vectors we have
\[
[e_\alpha,e_{-\alpha}] = h_\alpha\ ,\ [h_\alpha,e_\alpha] = 2e_\alpha\ ,\ [h_\alpha,e_{-\alpha}] = -2e_{-\alpha}.
\]
Since all root spaces are one-dimensional, there exists for all $\alpha,\beta \in \Delta$ with $\alpha+\beta \in \Delta$ numbers $N_{\alpha,\beta} \in \CC$ such that 
\[
[e_\alpha,e_\beta] = N_{\alpha,\beta}e_{\alpha+\beta}.
\]
We put $N_{\alpha,\beta}=0$ if $\alpha + \beta \notin \Delta$.
It is possible to choose the vectors $e_\alpha$ in such a way so that $N_{\alpha,\beta} = -N_{-\alpha,-\beta}$ holds. We put $h_\nu = h_{\alpha_\nu}$. Then the vectors $h_\nu,e_\alpha$ ($\nu \in \{1,\ldots,r\}$, $\alpha \in \Delta$) form a basis of $\cg^\CC$, a so-called Chevalley basis, with the properties
\begin{itemize}
\item[(1)] $[h_\nu,h_\mu] = 0$ for all $\nu,\mu \in \{1,\ldots,r\}$;
\item[(2)] $[h_\nu,e_\alpha] = \alpha(h_\nu)e_\alpha = c_{\alpha,\alpha_\nu}e_\alpha$ for all $\nu \in \{1,\ldots,r\}$ and $\alpha \in \Delta$;
\item[(3)] For all $\alpha \in \Delta$ there exists $c_1,\ldots,c_r \in \ZZ$ such that $[e_\alpha,e_{-\alpha}] = c_1h_1 + \ldots + c_rh_r$;
\item[(4)] For all $\alpha,\beta \in \Delta$ with $\alpha + \beta \neq 0$ we have 
$[e_\alpha,e_\beta] = N_{\alpha,\beta}e_{\alpha+\beta}$, where
\begin{itemize}
\item[(i)] $N_{\alpha,\beta} = 0$ if $\alpha + \beta \notin \Delta$;
\item[(ii)] $N_{\alpha,\beta} = \pm(p_{\alpha,\beta}+1)$ if $\alpha + \beta \in \Delta$ (see \cite{Sa90} about the sign ambiguity).
\end{itemize}
\end{itemize}
Some of the useful properties of the integers $N_{\alpha,\beta}$ are:
\begin{itemize}
\item[(1)] $N_{\alpha,\beta} = -N_{\beta,\alpha} = -N_{-\alpha,-\beta} = N_{\beta,-\alpha-\beta} = N_{-\alpha-\beta,\alpha}$ for all $\alpha,\beta \in \Delta$;
\item[(2)] $\frac{N_{\alpha,\beta}}{(\gamma,\gamma)} =  \frac{N_{\beta,\gamma}}{(\alpha,\alpha)} = \frac{N_{\gamma,\alpha}}{(\beta,\beta)}$ for all pairwise independent $\alpha,\beta,\gamma \in \Delta$ with $\alpha + \beta + \gamma = 0$;
\item[(3)] $N^2_{\alpha,\beta} = \frac{1}{2}(p_{\alpha,\beta}+1)q_{\alpha,\beta}(\alpha,\alpha)$;
\item[(4)] $N_{\delta,-\epsilon}N_{\gamma,\zeta-\gamma} + N_{-\epsilon,\gamma}N_{\delta,\zeta-\delta} = N_{\gamma,\delta}N_{-\epsilon,-\zeta}\frac{(\zeta,\zeta)}{(\eta,\eta)}$ for all $\gamma,\delta,\epsilon,\zeta \in \Delta^+$ with $\gamma + \delta = \epsilon + \zeta$ and $\delta \leq \zeta \leq \epsilon \leq \gamma$.
\end{itemize}
From (3) we deduce $p_{\alpha,\beta}+1 = \frac{1}{2}q_{\alpha,\beta}(\alpha,\alpha)$. In particular, $q_{\alpha,\beta} = \frac{2}{(\alpha,\alpha)} = B(e_\alpha,e_{-\alpha})$ if $p_{\alpha,\beta} = 0$.

For each $\alpha \in \Delta$ we now define
\[
u_\alpha = e_\alpha - e_{-\alpha}\ ,\ v_\alpha = i(e_\alpha + e_{-\alpha}).
\]
Then the compact real form $\cg$ of $\cg^\CC$ is given by
\[
\cg  = \ch \oplus \left( \bigoplus_{\alpha \in \Delta^+} (\RR u_\alpha \oplus \RR v_\alpha) \right).
\]
The Cartan decomposition $\cg = \ck \oplus \cp$ then obviously is given by
\[
\ck  = \ch \oplus \left( \bigoplus_{\alpha \in \Delta_K^+} (\RR u_\alpha \oplus \RR v_\alpha) \right)\  ,  \ 
\cp  = \bigoplus_{\alpha \in \Delta_M^+} (\RR u_\alpha \oplus \RR v_\alpha) ,
\]
where $\Delta_K^+ = \Delta_K \cap \Delta^+$ and $\Delta_M^+ = \Delta_M \cap \Delta^+$. The complex structure $J = \ad(iH^k)$ acts on $\cp \cong T_o\bar{M}$ by
\[ 
Ju_\alpha = v_\alpha\ ,\ Jv_\alpha = -u_\alpha\ \  (\alpha \in \Delta_M^+). 
\]
By defining 
\[
\CC u_\alpha = \RR u_\alpha \oplus \RR Ju_\alpha =  \RR u_\alpha \oplus \RR v_\alpha
\]
for $\alpha \in \Delta_M^+$, we can write
\[
\cp  = \bigoplus_{\alpha \in \Delta_M^+} \CC u_\alpha .
\]

The following equations will be used later without referring to them explicitly:
\begin{itemize}
\item[(1)] $[h,u_\alpha] = -i\alpha(h)v_\alpha$ for all $h \in \ch$ and $\alpha \in \Delta$;
\item[(2)] $[h,v_\alpha] = i\alpha(h)u_\alpha$ for all $h \in \ch$ and $\alpha \in \Delta$;
\item[(3)] $[u_\alpha,v_\alpha] = 2ih_\alpha$ for all $\alpha \in \Delta$;
\item[(4)] $[u_\alpha,u_\beta] = N_{\alpha,\beta}u_{\beta + \alpha} - N_{-\alpha,\beta}u_{\beta-\alpha}$ for all $\alpha,\beta \in \Delta$ with $\beta \neq \pm \alpha$;
\item[(5)] $[v_\alpha,v_\beta] = -N_{\alpha,\beta}u_{\beta + \alpha} - N_{-\alpha,\beta}u_{\beta-\alpha}$ for all $\alpha,\beta \in \Delta$ with $\beta \neq \pm \alpha$;
\item[(6)] $[u_\alpha,v_\beta] = N_{\alpha,\beta}v_{\beta + \alpha} - N_{-\alpha,\beta}v_{\beta-\alpha}$ for all $\alpha,\beta \in \Delta$ with $\beta \neq \pm \alpha$;
\item[(7)] $B(u_\alpha,u_\beta) = B(u_\alpha,v_\beta) = B(v_\alpha,v_\beta) = 0$ for all $\alpha,\beta \in \Delta$ with $\beta \neq \pm \alpha$; 
\item[(8)] $B(u_\alpha,u_\alpha) = B(v_\alpha,v_\alpha) = -2$ for all $\alpha \in \Delta$.
\end{itemize}
From (7) and (8) we see that the vectors $\frac{1}{\sqrt{2}}u_\alpha$, $\alpha \in \Delta_M^+$, provide a complex orthonormal basis of the tangent space $T_o\bar{M} \cong \cp$. 
For all $\alpha,\beta \in \Delta_M^+$, $\alpha + \beta \notin \Delta$ since $(\alpha + \beta)(H^k) = 2$ and the coefficient of $\alpha_k$ in the highest root $\delta$ is $1$. This implies 
\begin{itemize}
\item[(4')] $[u_\alpha,u_\beta] = - N_{-\alpha,\beta}u_{\beta-\alpha}$ for all $\alpha,\beta \in \Delta_M^+$;
\item[(5')] $[v_\alpha,v_\beta] = - N_{-\alpha,\beta}u_{\beta-\alpha}$ for all $\alpha,\beta \in \Delta_M^+$;
\item[(6')] $[u_\alpha,v_\beta] = - N_{-\alpha,\beta}v_{\beta-\alpha}$ for all $\alpha,\beta \in \Delta_M^+$.
\end{itemize}
Note that, since $\alpha + \beta \notin \Delta$, the $(-\alpha)$-string containing $\beta$ starts with $\beta$ and therefore 
\[
N_{-\alpha,\beta} = \pm 1.
\]

We finally list the root systems and corresponding Hermitian symmetric spaces. We also include the extended Dynkin diagrams for $\{\alpha_1,\ldots,\alpha_r,-\delta\}$.

\begin{itemize}
\item[($A_r$)] $V = \{v \in \RR^{r+1} : \langle v,e_1 + \ldots + e_{r+1} \rangle = 0\}$, $r \geq 1$;\\
$\Delta= \{e_\nu - e_\mu : \nu \neq \mu\}$; $\Delta^+ = \{e_\nu - e_\mu : \nu < \mu\}$;\\
$\alpha_1 = e_1 - e_2 , \ldots,\alpha_r = e_r - e_{r+1}$;\\
$\delta = \alpha_1 + \ldots + \alpha_r = e_1 - e_{r+1}$;\\
\[
\xy
\POS (0,0) *\cir<2pt>{} ="a",
(10,0) *\cir<2pt>{}="b",
(30,0) *\cir<2pt>{}="c",
(40,0) *\cir<2pt>{}="d",
(20,10) *{\times}="e",
(0,-5) *{\alpha_1},
(10,-5) *{\alpha_2},
(30,-5) *{\alpha_{r-1}},
(40,-5) *{\alpha_r},
\ar @{-} "a";"b",
\ar @{.} "b";"c",
\ar @{-} "c";"d",
\ar @{-} "a";"e",
\ar @{-} "d";"e",
\endxy
\]

\smallskip
\noindent For each $k \in \{1,\ldots,r\}$, the coefficient of $\alpha_k$ in $\delta$ is equal to one and the corresponding Hermitian symmetric space is the complex Grassmann manifold $G_k(\CC^{r+1}) = SU_{r+1}/S(U_kU_{r+1-k})$. Since $\alpha_k$ and $\alpha_{r+1-k}$ lead to isometric Grassmann manifolds $G_k(\CC^{r+1}) \cong G_{r+1-k}(\CC^{r+1})$, we will always assume $2k \leq r+1$. Then
\begin{align*}
\Delta_M^+  & = \{\alpha_\nu + \ldots + \alpha_\mu  : 1 \leq \nu \leq k \leq \mu \leq r\}\\
& = \{ e_\nu - e_{\mu+1} : 1 \leq \nu \leq k \leq \mu \leq r\}.
\end{align*}

\medskip
\item[($B_r$)] $V = \RR^r$, $r \geq 2$;\\
$\Delta = \{\pm e_\nu \pm e_\mu : \nu < \mu\} \cup \{\pm e_\nu\}$; $\Delta^+ = \{e_\nu \pm e_\mu : \nu < \mu\} \cup \{e_\nu\}$;\\
$\alpha_1 = e_1 - e_2 , \ldots,\alpha_{r-1} = e_{r-1} - e_r,\alpha_r=e_r$;\\
$\delta = \alpha_1 + 2\alpha_2 + \ldots + 2\alpha_r = e_1 + e_2$;\\
\[
\xy
\POS (0,0) *\cir<2pt>{} ="a",
(10,0) *\cir<2pt>{}="b",
(30,0) *\cir<2pt>{}="c",
(40,0) *\cir<2pt>{}="d",
(50,0) *\cir<2pt>{}="e",
(10,10) *{\times}="f",
(0,-5) *{\alpha_1},
(10,-5) *{\alpha_2},
(30,-5) *{\alpha_{r-2}},
(40,-5) *{\alpha_{r-1}},
(50,-5) *{\alpha_r},
\ar @{-} "a";"b",
\ar @{.} "b";"c",
\ar @{-} "c";"d",
\ar @2{->} "d";"e",
\ar @{-} "b";"f"
\endxy
\]

\medskip\noindent
The coefficient of $\alpha_1$ in $\delta$ is equal to one and the corresponding Hermitian symmetric space is the real Grassmann manifold $G_2^+(\RR^{2r+1}) = SO_{2r+1}/SO_{2r-1}SO_2$. Then
\begin{align*}
\Delta_M^+  & = \{\alpha_1 + \ldots + \alpha_\mu  : 1 \leq \mu \leq r\} \\
& \qquad  \cup \{\alpha_1 + \ldots + \alpha_\mu + 2\alpha_{\mu+1} + \ldots + 2\alpha_r : 1 \leq \mu < r\} \\
& = \{ e_1 \pm e_{\mu+1} : 1 \leq \mu < r\} \cup \{ e_1\} .
\end{align*}

\medskip
\item[($C_r$)] $V = \RR^r$, $r \geq 3$;\\
$\Delta = \{\pm e_\nu \pm e_\mu : \nu < \mu\} \cup \{\pm 2e_\nu\}$; $\Delta^+ = \{e_\nu \pm e_\mu : \nu < \mu\} \cup \{2e_\nu\}$;\\
$\alpha_1 = e_1 - e_2 , \ldots,\alpha_{r-1} = e_{r-1} - e_r,\alpha_r=2e_r$;\\
$\delta = 2\alpha_1 + \ldots + 2\alpha_{r-1} + \alpha_r = 2e_1$;\\
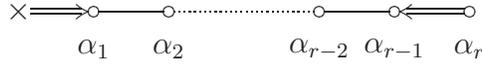
\[
\xy
\POS (0,0) *\cir<2pt>{} ="a",
(10,0) *\cir<2pt>{}="b",
(30,0) *\cir<2pt>{}="c",
(40,0) *\cir<2pt>{}="d",
(50,0) *\cir<2pt>{}="e",
(-10,0) *{\times}="f",
(0,-5) *{\alpha_1},
(10,-5) *{\alpha_2},
(30,-5) *{\alpha_{r-2}},
(40,-5) *{\alpha_{r-1}},
(50,-5) *{\alpha_r},
\ar @{-} "a";"b",
\ar @{.} "b";"c",
\ar @{-} "c";"d",
\ar @2{<-} "d";"e",
\ar @2{->} "f";"a"
\endxy
\]

\medskip\noindent
The coefficient of $\alpha_r$ in $\delta$ is equal to one and the corresponding Hermitian symmetric space is $Sp_r/U_r$. Then
\begin{align*}
\Delta_M^+  & =\{ \alpha_\nu + \ldots + \alpha_{\mu-1} + 2\alpha_\mu + \ldots + 2\alpha_{r-1} + \alpha_r  : 1 \leq \nu < \mu \leq r\} 
\\ & \qquad \cup \{2\alpha_\nu + \ldots + 2\alpha_{r-1} + \alpha_r : 1 \leq \nu \leq r\}\\
& =\{ e_\nu + e_\mu : 1 \leq \nu < \mu \leq r\}  \cup \{ 2e_\nu: 1 \leq \nu \leq r\}.
\end{align*}

\medskip
\item[($D_r$)] $V = \RR^r$, $r \geq 4$;\\
$\Delta = \{\pm e_\nu \pm e_\mu : \nu < \mu\}$; $\Delta^+ = \{e_\nu \pm e_\mu : \nu < \mu\}$;\\
$\alpha_1 = e_1 - e_2 , \ldots,\alpha_{r-1} = e_{r-1} - e_r,\alpha_r=e_{r-1}+e_r$;\\
$\delta = \alpha_1 + 2\alpha_2 + \ldots + 2\alpha_{r-2} + \alpha_{r-1} + \alpha_r = e_1 + e_2$;\\
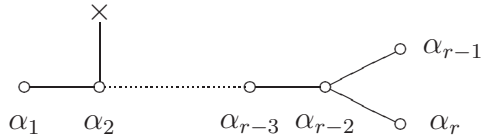
\[
\xy
\POS (0,0) *\cir<2pt>{} ="a",
(10,0) *\cir<2pt>{}="b",
(30,0) *\cir<2pt>{}="c",
(40,0) *\cir<2pt>{}="d",
(50,5) *\cir<2pt>{}="e",
(50,-5) *\cir<2pt>{}="f",
(10,10) *{\times}="g",
(0,-5) *{\alpha_1},
(10,-5) *{\alpha_2},
(30,-5) *{\alpha_{r-3}},
(40,-5) *{\alpha_{r-2}},
(57,5) *{\alpha_{r-1}},
(56,-5) *{\alpha_r},
\ar @{-} "a";"b",
\ar @{.} "b";"c",
\ar @{-} "c";"d",
\ar @{-} "d";"e",
\ar @{-} "d";"f",
\ar @{-} "b";"g"
\endxy
\]

\medskip\noindent
The coefficient of $\alpha_1$ in $\delta$ is equal to one and the corresponding Hermitian symmetric space is  the real Grassmann manifold $G_2^+(\RR^{2r}) = SO_{2r}/SO_{2r-2}SO_2$. Then
\begin{align*}
\Delta_M^+  & = \{\alpha_1 + \ldots + \alpha_\mu  : 1 \leq \mu \leq r\} 
 \cup  \{\alpha_1 + \ldots + \alpha_{r-2} + \alpha_r \} \\
& \qquad  \cup \{\alpha_1 + \ldots + \alpha_{\mu-1} + 2\alpha_\mu + \ldots + 2\alpha_{r-2} + \alpha_{r-1} + \alpha_r : 2 \leq \mu \leq r-2\}\\
& = \{ e_1 \pm e_\mu : 2 \leq \mu \leq r\} .
\end{align*}

The coefficient of $\alpha_r$ in $\delta$ is equal to one and the corresponding Hermitian symmetric space is $SO_{2r}/U_r$. (The Dynkin diagram symmetry implies that $\alpha_{r-1}$ leads to an isometric copy of $SO_{2r}/U_r$ and so we omit this case.) Then
\begin{align*}
\Delta_M^+  & =  \{\alpha_\nu + \ldots + \alpha_{r-2} + \alpha_r  : 1 \leq \nu \leq r-2\} \cup \{\alpha_r\}\\
& \qquad \cup \{\alpha_\nu + \ldots + \alpha_{\mu-1} + 2\alpha_\mu + \ldots + 2\alpha_{r-2} + \alpha_{r-1} + \alpha_r  : 1 \leq \nu < \mu \leq r-1\} \\
& = \{ e_\nu + e_\mu : 1 \leq \nu < \mu \leq r\} .
\end{align*}

\medskip
\item[($E_6$)] $V = \{v \in \RR^8 : \langle v,e_6-e_7 \rangle = \langle v , e_7 + e_8 \rangle = 0\}$;\\
$\Delta = \{\pm e_\nu \pm e_\mu : \nu < \mu \leq 5\} \cup \{ \frac{1}{2} \sum_{\nu=1}^8 (-1)^{n(\nu)}e_\nu \in V : \sum_{\nu=1}^8 n(\nu) \ {\rm even}\}$;\\
$\Delta^+ = \{e_\nu \pm e_\mu : \nu > \mu\} \cup \{ \frac{1}{2} (\sum_{\nu=1}^5 (-1)^{n(\nu)}e_\nu-e_6-e_7+e_8) : \sum_{\nu=1}^5 n(\nu) \ {\rm even}\}$;\\
$\alpha_1 = \frac{1}{2}(e_1-e_2-e_3-e_4-e_5-e_6-e_7+e_8)$,\\ 
$\alpha_2 = e_1 + e_2,\alpha_3 = e_2 - e_1,\alpha_4 = e_3 - e_2,\alpha_5 = e_4 - e_3,\alpha_6 = e_5 - e_4$;\\
$\delta = \alpha_1 + 2\alpha_2 + 2\alpha_3 + 3\alpha_4 + 2\alpha_5 + \alpha_6 = \frac{1}{2}(e_1 + e_2 + e_3 + e_4 + e_5 - e_6 - e_7 + e_8)$;\\
\[
\xy
\POS (0,0) *\cir<2pt>{} ="a",
(20,10) *\cir<2pt>{} = "b",
(10,0) *\cir<2pt>{}="c",
(20,0) *\cir<2pt>{}="d",
(30,0) *\cir<2pt>{}="e",
(40,0) *\cir<2pt>{}="f",
(20,20) *{\times}="g",
(0,-5) *{\alpha_1},
(25,10) *{\alpha_2},
(10,-5) *{\alpha_3},
(20,-5) *{\alpha_4},
(30,-5) *{\alpha_5},
(40,-5) *{\alpha_6},
\ar @{-} "a";"c",
\ar @{-} "c";"d",
\ar @{-} "b";"d",
\ar @{-} "d";"e",
\ar @{-} "e";"f",
\ar @{-} "b";"g"
\endxy
\]

\medskip\noindent
The coefficient of $\alpha_6$ in $\delta$ is equal to one and the corresponding Hermitian symmetric space is $E_6/Spin_{10}U_1$. (The Dynkin diagram symmetry implies that $\alpha_1$ leads to an isometric copy of $E_6/Spin_{10}U_1$ and so we omit this case.)  The set $\Delta_M^+$ consists of the following $16$ positive roots:
{\small
\[
\begin{matrix}
& & 0 & & \\
0 & 0 & 0 & 0 & 1
\end{matrix}
\qquad
\begin{matrix}
& & 0 & & \\
0 & 0 & 0 & 1 & 1
\end{matrix}
\qquad
\begin{matrix}
& & 0 & & \\
0 & 0 & 1 & 1 & 1
\end{matrix}
\qquad
\begin{matrix}
& & 0 & & \\
0 & 1 & 1 & 1 & 1
\end{matrix}
\]
\[ 
\begin{matrix}
& & 1 & & \\
0 & 0 & 1 & 1 & 1
\end{matrix}
\qquad
\begin{matrix}
& & 1 & & \\
0 & 1 & 1 & 1 & 1
\end{matrix}
\qquad 
\begin{matrix}
& & 1 & & \\
0 & 1 & 2 & 1 & 1
\end{matrix}
\qquad
\begin{matrix}
& & 1 & & \\
0 & 1 & 2 & 2 & 1
\end{matrix}
\]
\[
\begin{matrix}
& & 0 & & \\
1 & 1 & 1 & 1 & 1
\end{matrix}
\qquad
\begin{matrix}
& & 1 & & \\
1 & 1 & 1 & 1 & 1
\end{matrix}
\qquad
\begin{matrix}
& & 1 & & \\
1 & 1 & 2 & 1 & 1
\end{matrix}
\qquad
\begin{matrix}
& & 1 & & \\
1 & 2 & 2 & 1 & 1
\end{matrix}
\]
\[
\begin{matrix}
& & 1 & & \\
1 & 1 & 2 & 2 & 1
\end{matrix}
\qquad
\begin{matrix}
& & 1 & & \\
1 & 2 & 2 & 2 & 1
\end{matrix}
\qquad
\begin{matrix}
& & 1 & & \\
1 & 2 & 3 & 2 & 1
\end{matrix}
\qquad
\begin{matrix}
& & 2 & & \\
1 & 2 & 3 & 2 & 1
\end{matrix}
\]
}

Equivalently, $\Delta_M^+$ consists of
\begin{align*}
\Delta_M^+ & = \{e_5 \pm e_\mu: 5 > \mu\} \\
& \qquad \cup \left\{ \frac{1}{2} \left(\sum_{\nu=1}^4 (-1)^{n(\nu)}e_\nu+e_5-e_6-e_7+e_8\right) : \sum_{\nu=1}^4 n(\nu) \ {\rm even}\right\}.
\end{align*}

\medskip
\item[($E_7$)] $V = \{v \in \RR^8 : \langle v , e_7 + e_8 \rangle = 0\}$;\\
$\Delta = \{\pm e_\nu \pm e_\mu : \nu < \mu \leq 6\} \cup \{ \pm (e_7-e_8)\} \\ 
{} \hspace{30pt} \cup \{ \frac{1}{2} \sum_{\nu=1}^8 (-1)^{n(\nu)}e_\nu \in V : \sum_{\nu=1}^8 n(\nu) \ {\rm even}\}$;\\
$\Delta^+ = \{e_\nu \pm e_\mu : \nu > \mu\} \cup \{e_8 - e_7\} \cup \{ \frac{1}{2} (\sum_{\nu=1}^6 (-1)^{n(\nu)}e_\nu -e_7+e_8) : \sum_{\nu=1}^6 n(\nu) \ {\rm odd}\}$;\\
$\alpha_1 = \frac{1}{2}(e_1-e_2-e_3-e_4-e_5-e_6-e_7+e_8),\alpha_2 = e_1 + e_2$,\\
$\alpha_3 = e_2 - e_1,\alpha_4 = e_3 - e_2,\alpha_5 = e_4 - e_3,\alpha_6 = e_5 - e_4,\alpha_7 = e_6 - e_5$;\\
$\delta = 2\alpha_1 + 2\alpha_2 + 3\alpha_3 + 4\alpha_4 + 3\alpha_5 + 2\alpha_6 + \alpha_7 = e_8 - e_7$;\\
\[
\xy
\POS (0,0) *\cir<2pt>{} ="a",
(20,10) *\cir<2pt>{} = "b",
(10,0) *\cir<2pt>{}="c",
(20,0) *\cir<2pt>{}="d",
(30,0) *\cir<2pt>{}="e",
(40,0) *\cir<2pt>{}="f",
(50,0) *\cir<2pt>{}="g",
(-10,0) *{\times}="h",
(0,-5) *{\alpha_1},
(25,10) *{\alpha_2},
(10,-5) *{\alpha_3},
(20,-5) *{\alpha_4},
(30,-5) *{\alpha_5},
(40,-5) *{\alpha_6},
(50,-5) *{\alpha_7},
\ar @{-} "a";"c",
\ar @{-} "c";"d",
\ar @{-} "b";"d",
\ar @{-} "d";"e",
\ar @{-} "e";"f",
\ar @{-} "f";"g",
\ar @{-} "h";"a"
\endxy
\]

\medskip\noindent
The coefficient of $\alpha_7$ in $\delta$ is equal to one and the corresponding Hermitian symmetric space is $E_7/E_6U_1$. The set $\Delta_M^+$ consists of the following $27$ positive roots:
{\footnotesize
\[
\begin{matrix}
& & 0 & & & \\
0 & 0 & 0 & 0 & 0 & 1
\end{matrix}
\qquad
\begin{matrix}
& & 0 & & & \\
0 & 0 & 0 & 0 & 1 & 1
\end{matrix}
\qquad
\begin{matrix}
& & 0 & & & \\
0 & 0 & 0 & 1 & 1 & 1
\end{matrix}
\qquad
\begin{matrix}
& & 0 & & & \\
0 & 0 & 1 & 1 & 1 & 1
\end{matrix}
\]
\[
\begin{matrix}
& & 0 & & & \\
0 & 1 & 1 & 1 & 1 & 1
\end{matrix}
\qquad
\begin{matrix}
& & 1 & & & \\
0 & 0 & 1 & 1 & 1 & 1
\end{matrix}
\qquad
\begin{matrix}
& & 1 & & & \\
0 & 1 & 1 & 1 & 1 & 1
\end{matrix}
\qquad
\begin{matrix}
& & 1 & & & \\
0 & 1 & 2 & 1 & 1 & 1
\end{matrix}
\]
\[
\begin{matrix}
& & 1 & & & \\
0 & 1 & 2 & 2 & 1 & 1
\end{matrix}
\qquad
\begin{matrix}
& & 1 & & & \\
0 & 1 & 2 & 2 & 2 & 1
\end{matrix}
\qquad
\begin{matrix}
& & 0 & & & \\
1 & 1 & 1 & 1 & 1 & 1
\end{matrix}
\qquad
\begin{matrix}
& & 1 & & & \\
1 & 1 & 1 & 1 & 1 & 1
\end{matrix}
\]
\[
\begin{matrix}
& & 1 & & & \\
1 & 1 & 2 & 1 & 1 & 1
\end{matrix}
\qquad
\begin{matrix}
& & 1 & & & \\
1 & 1 & 2 & 2 & 1 & 1
\end{matrix}
\qquad
\begin{matrix}
& & 1 & & & \\
1 & 2 & 2 & 1 & 1 & 1
\end{matrix}
\qquad
\begin{matrix}
& & 1 & & & \\
1 & 1 & 2 & 2 & 2 & 1
\end{matrix}
\]
\[
\begin{matrix}
& & 1 & & & \\
1 & 2 & 2 & 2 & 1 & 1
\end{matrix}
\qquad
\begin{matrix}
& & 1 & & & \\
1 & 2 & 2 & 2 & 2 & 1
\end{matrix}
\qquad
\begin{matrix}
& & 1 & & & \\
1 & 2 & 3 & 2 & 1 & 1
\end{matrix}
\qquad
\begin{matrix}
& & 1 & & & \\
1 & 2 & 3 & 2 & 2 & 1
\end{matrix}
\]
\[
\begin{matrix}
& & 2 & & & \\
1 & 2 & 3 & 2 & 1 & 1
\end{matrix}
\qquad
\begin{matrix}
& & 1 & & & \\
1 & 2 & 3 & 3 & 2 & 1
\end{matrix}
\qquad
\begin{matrix}
& & 2 & & & \\
1 & 2 & 3 & 2 & 2 & 1
\end{matrix}
\qquad
\begin{matrix}
& & 2 & & & \\
1 & 2 & 3 & 3 & 2 & 1
\end{matrix}
\]
\[
\begin{matrix}
& & 2 & & & \\
1 & 2 & 4 & 3 & 2 & 1
\end{matrix}
\qquad
\begin{matrix}
& & 2 & & & \\
1 & 3 & 4 & 3 & 2 & 1
\end{matrix}
\qquad
\begin{matrix}
& & 2 & & & \\
2 & 3 & 4 & 3 & 2 & 1
\end{matrix}
\qquad
\begin{matrix}
& &  &  & & & & & \\
 &  &  &  &  & & & &
\end{matrix}
\]
}

Equivalently, $\Delta_M^+$ consists of
\begin{align*}
\Delta_M^+ & = \{e_6 \pm e_\mu : 6 > \mu\} \cup \{e_8 - e_7\}  \\
& \qquad \cup \left\{ \frac{1}{2} \left(\sum_{\nu=1}^5 (-1)^{n(\nu)}e_\nu +e_6 -e_7+e_8\right) : \sum_{\nu=1}^5 n(\nu) \ {\rm odd}\right\}.
\end{align*}

\end{itemize}

\section{Proof of Theorem \ref{classification}}

Let $M$ be a connected orientable real hypersurface in an irreducible Hermitian symmetric space $\bar{M}$ of compact type. Let $G = I^o(\bar{M})$, $o \in M$ and $K = G_o$ the isotropy group of $G$ at $o$.
We denote by $\cg$ and $\ck$ the Lie algebras of $G$ and $K$, respectively. We continue using the notations that we introduced in Section \ref{StcHss}.

Let $\cg = \ck \oplus \cp$ be the Cartan decomposition of $\cg$. Recall that
$\cp  = \bigoplus_{\alpha \in \Delta_M^+} \CC u_\alpha$.
By construction, the highest root $\delta$ is in $\Delta_M^+$ and so $u_\delta \in \cp$.

\begin{lm}\label{abelian} 
There exists a subset $\Omega$ of $\Delta_M^+$ with $| \Omega | = {\rm rk}(\bar{M})$ and $\delta \in \Omega$ so that $|u_\alpha| = |u_\beta|$ for all $\alpha,\beta \in \Omega$ and 
\[
\ca = \bigoplus_{\alpha \in \Omega} \RR u_\alpha
\]
is a maximal abelian subspace of $\cp$.
\end{lm}

\begin{proof}
We construct $\Omega$ explicitly using the root systems that we introduced in Section \ref{StcHss}:
\begin{itemize}
\item[($A_r$)] For $\alpha_k$: $\Omega = \{e_1-e_{r+1},\ldots,e_k-e_{r+2-k} \}$;
\item[($B_r$)] For $\alpha_1$: $\Omega = \{e_1+e_2,e_1-e_2\}$;
\item[($C_r$)] For $\alpha_r$: $\Omega = \{2e_1,\ldots,2e_r\}$;
\item[($D_r$)] For $\alpha_1$: $\Omega = \{e_1+e_2,e_1-e_2\}$;\\
For $\alpha_r$: $\Omega = \{e_1+e_2,\ldots,e_{r-1} + e_r\}$ if $r$ is even and 
$\Omega = \{e_1+e_2,\ldots,e_{r-2} + e_{r-1}\}$ if $r$ is odd;
\item[($E_6$)] For $\alpha_6$: $\Omega = \{\frac{1}{2}(e_1 + e_2 + e_3 + e_4 + e_5 - e_6 - e_7 + e_8),e_5 - e_4\}$;
\item[($E_7$)] For $\alpha_7$: $\Omega = \{e_8-e_7,e_6-e_5,e_6+e_5\}$.
\end{itemize}
It is straightforward to verify the properties stated in the lemma.
\end{proof}

Since every tangent vector of $\bar{M}$ at $o$ lies in a maximal abelian subspace of $\cp$, and all maximal abelian subspaces of $\cp$ are congruent to each other, we can assume that $J\xi_o \in \ca$. Thus we can write
\[
J\xi_o = \sum_{\alpha \in \Omega} a_\alpha u_\alpha
\]
with some $a_\alpha \in \RR$.

\begin{lm}\label{accoeff}
If $x = \sum_{\alpha \in \Omega} c_\alpha u_\alpha \in \ca$ is perpendicular to $J\xi_o \in \ca$,
then $a_\alpha = 0$ or $c_\alpha = 0$ for each $\alpha \in \Omega$.
\end{lm}

\begin{proof}
Note that $x \in \ca \cap {\mathcal C}_o$. 
Since both $J\xi_o$ and $x$ are in $\ca$, we have $\bar{R}_{J\xi_o}x = 0$. Since $x \in {\mathcal C}_o$, we also have $\bar{R}_{J\xi_o}Jx = 0$ by Proposition \ref{IRF2}. This implies $\bar{R}_{\xi_o}x = 0$, as the curvature tensor in a K\"{a}hler manifold is $J$-invariant. We have
\[
\xi_o = -J(J\xi_o) = - \sum_{\alpha \in \Omega} a_\alpha Ju_\alpha = - \sum_{\alpha \in \Omega} a_\alpha v_\alpha
\]
and therefore
\[
0 = \bar{R}_{\xi_o}x
= -[[x,\xi_o],\xi_o] 
= -\sum_{\alpha,\beta,\gamma \in \Omega} c_\alpha a_\beta a_\gamma  [[u_\alpha,v_\beta],v_\gamma].
\]
We work out the right-hand side of this equation in more detail. Firstly, we have
\[
[u_\alpha,v_\beta] = 
\begin{cases}
0 & \mbox{if } \alpha \neq \beta, \\
2ih_\alpha & \mbox{if } \alpha = \beta.
\end{cases}
\]
Inserting this into the previous equation yields 
\[
0 = \sum_{\beta,\gamma \in \Omega} c_\beta a_\beta a_\gamma  [ih_\beta,v_\gamma].
\]
Next, we have
\[
[ih_\beta,v_\gamma]  = i\gamma(ih_\beta)u_\gamma = - \gamma(h_\beta)u_\gamma 
 = - 2 \frac{(\gamma,\beta)}{(\beta,\beta)}u_\gamma
= \begin{cases}
0  & \mbox{if } \beta \neq \gamma, \\
-2u_\gamma & \mbox{if } \beta = \gamma.
\end{cases}
\]
Inserting this into the previous equation implies $0 = \sum_{\beta \in \Omega} c_\beta a_\beta^2  u_\beta$.
Since the vectors $u_\beta$, $\beta \in \Omega$, are linearly independent, we get $c_\beta = 0$ or $a_\beta = 0$ for each $\beta \in \Omega$. 
\end{proof}

\begin{lm}\label{normalproptoroot}
$J\xi_o$ is proportional to $u_\alpha$ for some $\alpha \in \Omega$.
\end{lm}

\begin{proof}
Let $\Omega_0 = \{\alpha \in \Omega : a_\alpha \neq 0\}$ and define a subspace $\ca_0$ of $\ca$ by
\[
\ca_0 = \bigoplus_{\alpha \in \Omega_0} \RR u_\alpha.
\]
Let $d = \dim(\ca_0)$ and assume that $d \geq 2$. Then we can find a vector $0 \neq x \in \ca_0$ that is perpendicular to $J\xi_o$. Since $x \in \ca$, we can write $x = \sum_{\alpha \in \Omega} c_\alpha u_\alpha$. Moreover, since $x \in \ca_0$, we have $c_\alpha =0$ for all $\alpha \notin \Omega_0$. 

As $x \neq 0$, we must have $c_\alpha \neq 0$ for some $\alpha \in \Omega_0$. According to Lemma \ref{accoeff} this implies $a_\alpha = 0$, which contradicts $\alpha \in \Omega_0$. It follows that $d = 1$, which means that $J\xi_o$ is proportional to  $u_\alpha$ for some $\alpha \in \Omega$. 
\end{proof}

The Weyl group of $\Delta$ acts transitively on roots of the same length. By a suitable transformation in the Weyl group we can arrange that $\alpha \in \Omega$ ($\alpha$ as in Lemma \ref{normalproptoroot}) becomes the highest root $\delta$. We can thus assume, without loss of generality, that $J\xi_o  = \frac{1}{\sqrt{2}}u_\delta$.

\medskip
We now define
\[
\Delta_M^+(0) = \{\alpha \in \Delta_M^+ : \delta - \alpha \notin \Delta^+\cup \{0\}\}\ ,\ 
\Delta_M^+(1) = \{\alpha \in \Delta_M^+ : \delta - \alpha \in \Delta^+\},
\]
which gives the disjoint union
\[
\Delta_M^+ = \Delta_M^+(0) \cup \Delta_M^+(1) \cup \{\delta\}.
\]
Explicitly, we have
\begin{itemize}
\item[($A_r$)] $\Delta_M^+(0) = \emptyset$ if $k =1$;\\
$\Delta_M^+(0)  = \{e_\nu - e_{\mu+1} : 2 \leq \nu \leq k \leq \mu \leq r-1\}$ if $k \geq 2$;
\item[($B_r$)] $\Delta_M^+(0) = \{e_1 - e_2\}$;
\item[($C_r$)] $\Delta_M^+(0) =\{ e_\nu + e_\mu : 2 \leq \nu < \mu \leq r\} \cup \{2e_\nu: 2 \leq \nu \leq r\}$;
\item[($D_r$)] $\Delta_M^+(0) = \{e_1 - e_2\}$ if $\bar{M} = G_2^+(\RR^{2r})$;\\
$\Delta_M^+(0)  = \{ e_\nu + e_\mu : 3 \leq \nu < \mu \leq r\}$ if $\bar{M} = SO_{2r}/U_r$;
\item[($E_6$)] $\Delta_M^+(0)  
= \{e_5-e_4,e_5-e_3,e_5-e_2,e_5-e_1,\frac{1}{2}(-e_1-e_2-e_3-e_4+e_5-e_6-e_7+e_8)\}$;
\item[($E_7$)] $\Delta_M^+(0) = \{e_6 \pm e_j : 6 > j\}$.
\end{itemize}
The set $\Delta_M^+(1)$ can be worked out explicitly from $\Delta_M^+$ and $\Delta_M^+(0)$ by using $\Delta_M^+(1) = \Delta_M^+ \setminus (\Delta_M^+(0) \cup \{\delta\})$.

\medskip
These sets  arise naturally as eigenspaces of the Jacobi operator $\bar{R}_{J\xi_o}$.

\begin{lm}\label{nJo} The Jacobi operator $\bar{R}_{J\xi_o}$ is given by
\[
\bar{R}_{J\xi_o}x = 
\begin{cases}
0 & ,\ {\rm if}\ x \in \RR u_\delta \oplus \bigoplus_{\alpha \in \Delta_M^+(0)} \CC u_\alpha\\
\frac{1}{2}x & ,\ {\rm if}\ x \in \bigoplus_{\alpha \in \Delta_M^+(1)} \CC u_\alpha   \\
2x & ,\ {\rm if}\ x \in \RR v_\delta.
\end{cases}
\]
\end{lm}

\begin{proof}
Recall that the Riemannian curvature tensor $\bar{R}$ on $\bar{M}$ satisfies
\[
\bar{R}(x,y)z = -[[x,y],z]\ \ (x,y,z \in T_o\bar{M} \cong \cp).
\]
We have $\bar{R}_{u_\delta}u_\delta  = 0$, since $\bar{R}$ is skew-symmetric in the first two variables. Next,
\[
\bar{R}_{u_\delta}v_\delta = - [[v_\delta,u_\delta],u_\delta] = 2i[h_\delta,u_\delta] = 2\delta(h_\delta)v_\delta = 4v_\delta.
\]
If $\alpha \in \Delta_M^+(0)$, we have $\bar{R}_{u_\delta}u_\alpha = 0$ and $\bar{R}_{u_\delta}v_\alpha = 0$. Finally, if $\alpha \in \Delta_M^+(1)$, we get
\begin{align*}
\bar{R}_{u_\delta}u_\alpha & = -[[u_\alpha,u_\delta],u_\delta] = [[u_\delta,u_\alpha],u_\delta] = - N_{-\delta,\alpha}[u_{\alpha-\delta},u_\delta] = N_{-\delta,\alpha}[u_\delta,u_{\alpha-\delta}] \\
& = N_{-\delta,\alpha}N_{\delta,\alpha-\delta}u_\alpha. 
\end{align*}
The coefficient $p_{-\delta,\alpha}$ in the $(-\delta)$-string containing $\alpha$ must be zero, which implies $N_{-\delta,\alpha} = \pm 1$. Similarly, the coefficient $p_{\delta,\alpha-\delta}$ in the $\delta$-string containing $\alpha-\delta$ must be zero, which implies $N_{\delta,\alpha-\delta} = \pm 1$. Since $\bar{M}$ has nonnegative sectional curvature, we get $g(\bar{R}_{u_\delta}u_\alpha,u_\alpha) \geq 0$, which implies $N_{-\delta,\alpha}N_{\delta,\alpha-\delta} = 1$. It follows that $\bar{R}_{u_\delta}u_\alpha = u_\alpha$. Analogously, we can prove
\begin{align*}
\bar{R}_{u_\delta}v_\alpha & = -[[v_\alpha,u_\delta],u_\delta] = [[u_\delta,v_\alpha],u_\delta] = - N_{-\delta,\alpha}[v_{\alpha-\delta},u_\delta] = N_{-\delta,\alpha}[u_\delta,v_{\alpha-\delta}] \\
& =  N_{-\delta,\alpha}N_{\delta,\alpha-\delta}v_\alpha = v_\alpha.
\end{align*}
Taking into account that $J\xi_o = \frac{1}{\sqrt{2}}u_\delta$, we get the desired expression for $\bar{R}_{J\xi_o}$.
\end{proof}

We now give a geometric interpretation of the sets $\Delta_M^+(0)$ and $\Delta_M^+(1)$.
For this we decompose $T_o\bar{M} = \cp  = \bigoplus_{\alpha \in \Delta_M^+} \CC u_\alpha$ into
\[
\cp = \cp(0) \oplus \cp(1) \oplus \CC u_\delta
\]
with
\[
\cp(0) = \bigoplus_{\alpha \in \Delta_M^+(0)} \CC u_\alpha\ ,\ \cp(1) = \bigoplus_{\alpha \in \Delta_M^+(1)} \CC u_\alpha.
\]
Let $\ck(0) = [\cp(0),\cp(0)]$ and $\cg(0) = \ck(0) \oplus \cp(0)$. Explicitly, we have
\begin{itemize}
\item[($A_r$)] $\ck(0) = \{0\}$ if $k = 1$;\\
$\ck(0) = \RR i h_{\alpha_k} \oplus \cs\cu_{k-1} \oplus \cs\cu_{r-k} \cong \cu_1 \oplus \cs\cu_{k-1} \oplus \cs\cu_{r-k}$ if $k \geq 2$, where $\cs\cu_{k-1}$ is generated by the simple roots $\alpha_2,\ldots,\alpha_{k-1}$ and $\cs\cu_{r-k}$ is generated by the simple roots $\alpha_{k+1},\ldots,\alpha_{r-1}$;
\item[($B_r$)] $\ck(0) = \RR i h_{\alpha_1} \cong \cu_1$;
\item[($C_r$)] $\ck(0) = \RR i h_{\alpha_r} \oplus \cs\cu_{r-1} \cong \cu_{r-1}$, where $\cs\cu_{r-1}$ is generated by the simple roots $\alpha_2,\ldots,\alpha_{r-1}$;
\item[($D_r$)] $\ck(0) = \RR i h_{\alpha_1} \cong \cu_1$  if $\bar{M} = G_2^+(\RR^{2r})$;\\
$\ck(0) = \RR i h_{\alpha_r} \oplus \cs\cu_{r-2} \cong \cu_{r-2}$ if $\bar{M} = SO_{2r}/U_r$, where $\cs\cu_{r-2}$ is generated by the simple roots $\alpha_3,\ldots,\alpha_{r-1}$;
\item[($E_6$)] $\ck(0) = \RR i h_{\alpha_6} \oplus \cs\cu_5 \cong \cu_5$, where $\cs\cu_5$ is generated by the simple roots $\alpha_1,\alpha_3,\alpha_4,\alpha_5$;
\item[($E_7$)] $\ck(0) = \RR i h_{\alpha_7} \oplus \cs\co_{10} \cong \cs\co_2 \oplus \cs
\co_{10}$, where $\cs\co_{10}$ is generated by the simple roots $\alpha_2,\alpha_3,\alpha_4,\alpha_5,\alpha_6$;
\end{itemize}
and
\begin{itemize}
\item[($A_r$)] $\cg(0) = \{0\}$ if $k = 1$ and $\cg(0) \cong \cs\cu_{r-1}$ if $k \geq 2$;
\item[($B_r$)] $\cg(0) \cong \cs\cu_2$;
\item[($C_r$)] $\cg(0) \cong \cs\cp_{r-1}$;
\item[($D_r$)] $\cg(0) \cong \cs\cu_2$ if $\bar{M} = G_2^+(\RR^{2r})$;\\
$\cg(0) \cong \cs\co_{2r-4}$ if $\bar{M} = SO_{2r}/U_r$;
\item[($E_6$)] $\cg(0) \cong \cs\cu_6$;
\item[($E_7$)] $\cg(0) \cong \cs\co_{12}$.
\end{itemize}
The Cartan decomposition of the semisimple Lie algebra $\cg(0)$ is $\cg(0)  = \ck(0) \oplus \cp(0)$.
Let $G(0)$ be the connected closed subgroup of $G$ with Lie algebra $\cg(0)$. Since $\cg(0)  = (\cg(0) \cap \ck) \oplus (\cg(0) \cap \cp)$, the orbit $\Sigma(0) = G(0) \cdot o = G(0)/K(0)$ of $G(0)$ containing $o$ is a totally geodesic submanifold of $\bar{M}$. Explicitly, we have
\begin{itemize}
\item[($A_r$)] $\Sigma(0) \cong G_{k-1}(\CC^{r-1})$;
\item[($B_r$)] $\Sigma(0) \cong \CC P^1$;
\item[($C_r$)] $\Sigma(0) \cong Sp_{r-1}/U_{r-1}$;
\item[($D_r$)] $\Sigma(0) \cong \CC P^1$ if $\bar{M} = G_2^+(\RR^{2r})$;\\
$\Sigma(0) \cong SO_{2r-4}/U_{r-2}$ if $\bar{M} = SO_{2r}/U_r$;
\item[($E_6$)] $\Sigma(0) \cong \CC P^5$;
\item[($E_7$)] $\Sigma(0) \cong G_2^+(\RR^{12})$.
\end{itemize}
In particular, $\cp(0)$ is a Lie triple system in $\cp$. Obviously, $\CC u_\delta$ is a Lie triple system in $\cp$ and the corresponding totally geodesic submanifold of $\bar{M}$ is $\CC P^1$. The subspace $\CC u_\delta \oplus \cp(0)$ is also a Lie triple system in $\cp$ and the corresponding totally geodesic submanifold of $\bar{M}$ is $\CC P^1 \times \Sigma(0)$. Geometrically, $\CC P^1 \times \Sigma(0)$ is a meridian in $\bar{M}$ (see \cite{CN78}). The complementary subspace $\cp(1)$ is also a Lie triple system in $\cp$ and the corresponding totally geodesic submanifold $\Sigma(1)$ of $\bar{M}$ is a polar of $\bar{M}$ (see \cite{CN78}). Explicitly, we have
\begin{itemize}
\item[($A_r$)] $\Sigma(1) \cong \CC P^{k-1} \times \CC P^{r-k}$;
\item[($B_r$)] $\Sigma(1) \cong G_2^+(\RR^{2r-1})$;
\item[($C_r$)] $\Sigma(1) \cong \CC P^{r-1}$;
\item[($D_r$)]  $\Sigma(1) \cong G_2^+(\RR^{2r-2})$ if $\bar{M} = G_2^+(\RR^{2r})$;\\
$\Sigma(1) \cong G_2(\CC^r)$ if $\bar{M} = SO_{2r}/U_r$;
\item[($E_6$)] $\Sigma(1) \cong SO_{10}/U_5$;
\item[($E_7$)] $\Sigma(1) \cong E_6/Spin_{10}U_1$.
\end{itemize}
The sets $\Delta_M^+(0)$ and $\Delta_M^+(1)$ are therefore intimately related to a particular polar/meridian configuration in the Hermitian symmetric space $\bar{M}$. 

\medskip
We now use Jacobi field theory to investigate the structure of the focal sets of $M$ (see \cite{BCO16}, Section 10.2.2, for details about the methodology). From Corollary \ref{IRF7} we know that $M$ is curvature-adapted and thus we can diagonalize the shape operator $S$ and the normal Jacobi operator $\bar{R}_{J\xi_o}$ simultaneously. In particular, the maximal complex subbundle ${\mathcal C}$ of $TM$ is invariant under the shape operator $S$ of $M$. Note that
${\mathcal C}_o = \cp(0) \oplus \cp(1)$.

Let $X \in \cp(0)$ with $SX = x X$. Then we also have $S\phi X = x \phi X$, since $S\phi = \phi S$ by Proposition \ref{isomcomm}. From Corollary \ref{IRF6} we then get $x(x-a) = 0$. Next, let $X \in \cp(1)$ with $SX = x X$. From Corollary \ref{IRF6} we see that $2x^2 - 2ax - 1 = 0$. Since $a$ is constant by Proposition \ref{IRF5}, it follows that $M$ is a real hypersurface of $\bar{M}$ with constant principal curvatures. We write $a = \sqrt{2}\cot(\sqrt{2}t)$ with some $0 < t < \frac{\pi}{\sqrt{2}}$. Then the other possible principal curvatures of $M$ are $b =0$, $c = \frac{1}{\sqrt{2}}\cot(\frac{1}{\sqrt{2}}t)$ and $d = -\frac{1}{\sqrt{2}}\tan(\frac{1}{\sqrt{2}}t)$. Note that $c$ and $d$ are the two different solutions of the quadratic equation $2x^2 - 2a x - 1 = 0$. For $x \in \{a,b,c,d\}$ we define $T_x = \{X \in {\mathcal C} : SX = x X\}$. Then we have
\[
{\mathcal C} = T_a \oplus T_b \oplus T_c \oplus T_d\ ,\ \cp(0) = (T_a \oplus T_b)_o\ ,\ \cp(1) = (T_c \oplus T_d)_o
\]
and, if $T_x$ is not trivial, $T_x$ is the subbundle of $TM$ consisting of all principal curvature vectors of $M$ with respect to $x$ which are perpendicular to $\xi$. Since $S\phi = \phi S$, each $T_x$ is a complex subbundle of $TM$.

For $p \in M$ we denote by $\gamma_p$ the geodesic in $\bar{M}$ with $\gamma_p(0) = p$ and $\dot{\gamma}_p(0) = J\xi_p$, and by $f$ the
smooth map $f : M \to \bar{M} \ ,\ p \mapsto \gamma_p(t)$.
Geometrically, $f$ is the displacement of $M$ at distance $t$ in direction of the normal vector field $J\xi$. For each $p \in M$ the differential
$d_pf$ of $f$ at $p$ can be computed using Jacobi fields by means of
$d_pf(X) = Z_X(t)$,
where $Z_X$ is the Jacobi field along $\gamma_p$ with initial values $Z_X(0) = X$ and $Z_X^\prime(0) = -SX$. Using the explicit description of the Jacobi operator $\bar{R}_{J\xi_o}$ in Lemma \ref{nJo} (which holds at any point of $M$ by using a suitable conjugation) and of the shape operator $S$, we see that $(T_a \oplus T_b)_p$ is contained in the $0$-eigenspace of $\bar{R}_{J\xi_p}$ and $(T_c \oplus T_d)_p$ in the $\frac{1}{2}$-eigenspace of $\bar{R}_{J\xi_p}$. For the Jacobi fields along $\gamma_p$ we thus get the expressions
\[
Z_X(t) = \begin{cases}
\left(\cos(\sqrt{2}t)-\frac{a}{\sqrt{2}}\sin(\sqrt{2}t)\right)E_X(t)
& , \mbox{ if}\ X \in \RR \xi \\
\left(\cos(\frac{1}{\sqrt{2}}t) - x\sqrt{2}\sin(\frac{1}{\sqrt{2}}t)\right) E_X(t)
& , \mbox{ if}\ X \in T_x\ \mbox{and}\ x \in \{c,d\} \\
(1 - x t)E_X(t) & , \mbox{ if}\ X \in T_x\ \mbox{and}\ x \in \{a,b\}\ ,
\end{cases}
\]
where $E_X$ denotes the parallel vector field along $\gamma_p$ with $E_X(0) = X$. This shows that the kernel of $df$ is $\RR\xi \oplus T_c$
and that $f$ is of constant rank $\rk(T_a \oplus T_b \oplus T_d)$. So, locally, $f$ is a submersion onto a submanifold $P$ of $\bar{M}$. Moreover, the tangent space of $P$ at $f(p)$ is obtained by parallel translation of $(T_a \oplus T_b \oplus T_d)_p$, which is a complex subspace of $T_p\bar{M}$. Since $J$ is parallel along $\gamma_p$, $T_{f(p)}P$ is a complex subspace of $T_{f(p)}\bar{M}$. Thus $P$ is a complex submanifold of $\bar{M}$. 

The vector $\eta_p = \dot{\gamma}_p(t)$ is a unit normal vector of $P$ at $f(p)$. The shape operator $S^P_{\eta_p}$ of $P$ with respect to $\eta_p$ can be computed using Jacobi fields via the equation
$S^P_{\eta_p}Z_X(t) = - Z_X^\prime(t)$.
From this we immediately get that for each $x \in \{a,b,d\}$ the parallel translate of $(T_x)_p$ along $\gamma_p$ from $p$ to $f(p)$ is a principal curvature space of $P$ with respect to $\eta_p$, provided that $(T_x)_p$ is nontrivial. Moreover, the corresponding principal curvature is $0$ for $x \in \{b,d\}$ and $\frac{\sqrt{2}\cot(\sqrt{2}t)}{1-\sqrt{2}\cot(\sqrt{2}t)t}$ for $x = a$. Any complex submanifold of a K\"ahler
manifold is minimal and thus $T_a$ is either trivial or $t = \frac{\pi}{\sqrt{8}}$ (in which case the corresponding principal curvature becomes zero). The vectors of the form $\eta_q$, $q \in f^{-1}(\{f(p)\})$, form an open subset of the unit sphere in the normal space of $P$ at $f(p)$. Since $S^P_{\eta_q}$ vanishes for all $\eta_q$ it follows that $P$ is totally geodesic in $\bar{M}$. As $M$ is connected, rigidity of totally geodesic submanifolds implies that the entire submanifold $M$ is an open part of a tube with radius $t$ around a connected, complete, totally geodesic, complex submanifold $P$ of $\bar{M}$. 

The principal curvatures of $M$ with respect to the normal vector field $-J\xi$ are of course $-a,-b,-c,-d$. The corresponding eigenspace distributions do not change and therefore coincide with $T_a,T_b,T_c,T_d$. Now define $s = \frac{\pi}{\sqrt{2}} - t$. Then we have $0 < s < \frac{\pi}{\sqrt{2}}$, $-a = \sqrt{2}\cot(\sqrt{2}s)$, $-b = 0$, $-c = -\frac{1}{\sqrt{2}}\tan(\frac{1}{\sqrt{2}}s)$ and $-d = \frac{1}{\sqrt{2}}\cot(\frac{1}{\sqrt{2}}s)$. With exactly the same arguments as above, just interchanging the roles of $c$ and $d$, we can show that the entire submanifold $M$ is an open part of a tube with radius $s$ around a connected, complete, totally geodesic, complex submanifold $Q$ of $\bar{M}$.

We summarize the previous discussion in
\begin{prop}\label{focalsets}
Let $0 < t < \frac{\pi}{\sqrt{2}}$ so that $a = \sqrt{2}\cot(\sqrt{2}t)$ and define $s = \frac{\pi}{\sqrt{2}} - t$. Then there exist connected, complete, complex, totally geodesic submanifolds $P$ and $Q$ of $\bar{M}$ so that $M$ is an open part of a tube with radius $t$ around $P$ and an open part of a tube with radius $s$ around $Q$. Moreover, $p = \gamma_o(t) \in P$, $q = \gamma_o(-s) \in Q$,
\[
T_pP = \cp(0) \oplus (T_d)_o\ ,\ \nu_pP = \CC u_\delta \oplus (T_c)_o
\]
and
\[
T_qQ = \cp(0) \oplus (T_c)_o\ ,\ \nu_qQ = \CC u_\delta \oplus (T_d)_o,
\]
where we identify tangent vectors along the geodesic $\gamma_o$ by parallel translation.
\end{prop}

Since $\cp(0)$, $\cp(0) \oplus (T_d)_o$ and $\cp(0) \oplus (T_c)_o$ are Lie triple system in $\cp$, the subspace $\cp(0)$ is also a Lie triple system in both $\cp(0) \oplus (T_d)_o$ and $\cp(0) \oplus (T_c)_o$. This means that there exist isometric copies of $\Sigma(0) = G(0)/K(0)$ in both $P$ and $Q$ and in both cases the embedding is totally geodesic. The slice representation of $K(0)$ on the two normal spaces $\nu_pP$ and $\nu_qQ$ leaves $\CC u_\delta$ invariant and hence also $(T_c)_o$ and $(T_d)_o$. The representation of $\ck(0)$ on $\cp(1) = (T_c \oplus T_d)_o$ is as follows:
\begin{itemize}
\item[($A_r$)] trivial representation of $\{0\}$ on $\CC^{r-1} \cong \cp(1)$ if $k = 1$;\\
standard representation of $\cu_1 \oplus \cs\cu_{k-1} \oplus \cs\cu_{r-k}$ on $\CC^{k-1} \oplus \CC^{r-k} \cong \CC^{r-1} \cong \cp(1)$ if $k \geq 2$;
\item[($B_r$)] standard representation of $\cu_1$ on $\bigoplus^{2r-3}\CC \cong \cp(1)$;
\item[($C_r$)] standard representation of $\cu_{r-1}$ on $\CC^{r-1} \cong \cp(1)$;
\item[($D_r$)] standard representation of $\cu_1$ on $\bigoplus^{2r-4}\CC \cong \cp(1)$ if $\bar{M} = G_2^+(\RR^{2r})$;\\ 
standard representation of $\cu_{r-2}$ on $\CC^{r-2} \oplus \CC^{r-2} \cong \CC^{2r-4} \cong \cp(1)$ if $\bar{M} = SO_{2r}/U_r$;
\item[($E_6$)] standard representation of $\cu_5$ on $\Lambda^2\CC^5 \cong \CC^{10} \cong \cp(1)$;
\item[($E_7$)] the irreducible representation of $\cs\co_2 \oplus \cs\co_{10}$ on $\CC^{16} \cong \cp(1)$ that is induced from the isotropy representation of $E_6/Spin_{10}U_1$. This involves the two inequivalent irreducible spin representations of $\cs\co_{10}$. We only need the fact that this representation is irreducible, but the interested reader can find an explicit version in \cite{Ad96}.
\end{itemize}

\medskip
The index $i(\bar{M})$ of a Riemannian manifold $\bar{M}$ is the minimal codimension of a (nontrivial) totally geodesic submanifold. For Riemannian symmetric spaces the index was first studied by Onishchik in \cite{On80} and then by the first author and Olmos in \cite{BO16} and \cite{BO17}. In particular, we have

\begin{thm} \label{index2} {\rm (\cite{BO17},\cite{On80})}
The index of an irreducible Hermitian symmetric space $\bar{M}$ of compact type satisfies $i(\bar{M}) \geq 2$. Moreover, the equality $i(\bar{M}) = 2$ holds if and only if $\bar{M} = \CC P^r$ or $\bar{M} = G_2^+(\RR^{2r+1})$ or $\bar{M} = G_2^+(\RR^{2r})$.
\end{thm}

The Hermitian symmetric spaces with $i(\bar{M}) = 2$ admit totally geodesic complex hypersurfaces, namely $\CC P^{r-1} \subset \CC P^r$, $G_2^+(\RR^{2r}) \subset G_2^+(\RR^{2r+1})$ and $G_2^+(\RR^{2r-1}) \subset G_2^+(\RR^{2r})$. Theorem \ref{index2} implies that the irreducible Hermitian symmetric spaces of compact type that admit a totally geodesic complex hypersurface are, in our notation, those for which the root system and corresponding simple root are $((A_r),\alpha_1)$, $((B_r),\alpha_1)$ and $((D_r),\alpha_1)$. The classification of real hypersurfaces with isometric Reeb flow in these Hermitian symmetric spaces was obtained by Okumura in \cite{Ok75} and the authors in \cite{BS13}:

\begin{thm} \label{known} {\rm (\cite{BS13},\cite{Ok75})}
Let $M$ be a connected orientable real hypersurface with isometric Reeb flow in $\bar{M} = \CC P^r$ or $\bar{M} = G_2^+(\RR^{2r})$. Then $M$ is congruent to an open part of a tube of radius $0 < t < \pi/\sqrt{2}$ around the totally geodesic submanifold $\Sigma$ in $\bar{M}$, where
\begin{itemize}
\item[(i)] $\bar{M} = \CC P^r$ and $\Sigma = \CC P^k$, $0 \leq k \leq r-1$;
\item[(ii)] $\bar{M} = G_2^+(\RR^{2r})$ and $\Sigma = \CC P^{r-1}$, $3 \leq r$.
\end{itemize}
There exist no real hypersurfaces with isometric Reeb flow in $G_2^+(\RR^{2r+1})$, $r \geq 2$.
\end{thm}

We assume from now on that $i(\bar{M}) > 2$, or equivalently, that $\bar{M}$ does not admit a totally geodesic complex hypersurface. We separate our argument into two cases.

\medskip

{\sc Case 1.} {\it The representation of $\ck(0)$ on $\cp(1)$ is irreducible.}

\smallskip
The representation of $\ck(0)$ on $\cp(1)$ is irreducible precisely if the root system is $(C_r)$ ($r \geq 3$), $(E_6)$ or $(E_7)$. Since $[\ck(0),\cp(0)] \subset \cp(0)$, it follows using Proposition \ref{focalsets} that $[\ck(0), (T_c)_o] \subset (T_c)_o$ and $[\ck(0), (T_d)_o] \subset (T_d)_o$. Since $\cp(1) = (T_c \oplus T_d)_o$ and the representation of $\ck(0)$ on $\cp(1)$ is irreducible, either $(T_c)_o$ or $(T_d)_o$ must be trivial. This implies that either $P$ or $Q$ is a totally geodesic complex hypersurface of $\bar{M}$. This contradicts $i(\bar{M}) > 2$. We thus conclude:

\begin{thm}\label{nonexistence}
There exist no real hypersurfaces with isometric Reeb flow in the Hermitian symmetric spaces $Sp_r/U_r$ ($r \geq 3$), $E_6/Spin_{10}U_1$ and $E_7/E_6U_1$. 
\end{thm}

\medskip

{\sc Case 2.} {\it The representation of $\ck(0)$ on $\cp(1)$ is reducible.}

\smallskip
The representation of $\ck(0)$ on $\cp(1)$ is reducible precisely if the root system and corresponding simple root is $((A_r),\alpha_k)$ ($4 \leq 2k \leq r+1$) or $((D_r),\alpha_r)$ ($r \geq 4$). Note that the index of $G_2(\CC^4) \cong G_2^+(\RR^6)$ and of $SO_8/U_4 \cong G_2^+(\RR^8)$ is equal to $2$ and so we exclude these low-dimensional cases here. 

As above, we know that $[\ck(0), (T_c)_o] \subset (T_c)_o$ and $[\ck(0), (T_d)_o] \subset (T_d)_o$. However, the representation of $\ck(0)$ on $\cp(1)$ has exactly two irreducible components in both cases, which must coincide with $(T_c)_o$ and $(T_d)_o$. This allows us to work out $P$ and $Q$ explicitly. 

We start with $((A_r),\alpha_k)$ ($k \geq 2$). In this case we have 
\[
\Delta_M^+(0)  = \{e_\nu - e_{\mu+1}  : 2 \leq \nu \leq k \leq \mu \leq r-1\}
\]
and
\[
\Delta_M^+(1)  = \{e_\nu - e_{r+1}  : 2 \leq \nu \leq k\} \cup \{e_1 - e_{\mu+1}  : k \leq \mu \leq r-1\}.
\]
These two subsets of $\Delta_M^+(1)$ induce the reducible decomposition 
\[
\cp(1) = \bigoplus_{\alpha \in \Delta_M^+(1)} \CC u_\alpha \cong \CC^{k-1} \oplus \CC^{r-k}.
\]
The roots 
\[
\Delta_M^+(0) \cup  \{e_\nu - e_{r+1}  : 2 \leq \nu \leq k\} = \{e_\nu - e_{\mu+1}  : 2 \leq \nu \leq k \leq \mu \leq r\}
\]
and 
\[
\Delta_M^+(0) \cup \{e_1 - e_{\mu+1}  : k \leq \mu \leq r-1\} = \{e_\nu - e_{\mu+1} : 1 \leq \nu \leq k \leq \mu \leq r-1\}
\]
define Lie triple systems in $\cp$ for which the corresponding totally geodesic submanifolds are $G_{k-1}(\CC^r)$ and $G_k(\CC^r)$, respectively. Since totally geodesic submanifolds are uniquely determined by their Lie triple systems, we see that $P$ and $Q$ coincide with $G_{k-1}(\CC^r)$ and $G_k(\CC^r)$ (in no particular order). Thus we proved:

\begin{thm}\label{Grassmann}
Let $M$ be a real hypersurfaces with isometric Reeb flow in the complex Grassmann manifold $G_k(\CC^{r+1})$, $4 \leq 2k \leq r+1$, $(k,r) \neq (2,3)$. Then $M$ is congruent to an open part of a tube around the totally geodesic submanifold $G_k(\CC^r)$ in $G_k(\CC^{r+1})$.
\end{thm}

\medskip
Next, we consider $((D_r),\alpha_r)$. In this case we have 
\[
\Delta_M^+(0)  =  \{e_\nu + e_\mu : 3 \leq \nu < \mu \leq r\}
\]
and
\[
\Delta_M^+(1)  = \{e_1 + e_\mu : 2 < \mu \leq r\} \cup \{e_2 + e_\mu : 2 < \mu \leq r\}.
\]
These two subsets of $\Delta_M^+(1)$ induce the reducible decomposition 
\[
\cp(1) = \bigoplus_{\alpha \in \Delta_M^+(1)} \CC u_\alpha \cong \CC^{r-2} \oplus \CC^{r-2}.
\]
The roots 
\[
\Delta_M^+(0) \cup  \{e_1 + e_\mu : 2 < \mu \leq r\} 
\mbox{ and }
\Delta_M^+(0) \cup \{e_2 + e_\mu : 2 < \mu \leq r\} 
\]
define Lie triple systems in $\cp$ for which the corresponding totally geodesic submanifolds are isometric copies of $SO_{2r-2}/U_{r-1}$. We conclude that $P$ and $Q$ both coincide with a totally geodesic $SO_{2r-2}/U_{r-1}$. Thus we proved:

\begin{thm}\label{SO2rUr}
Let $M$ be a real hypersurfaces with isometric Reeb flow in the Hermitian symmetric space $SO_{2r}/U_r$ and $r \geq 5$. Then $M$ is congruent to an open part of a tube around the totally geodesic submanifold $SO_{2r-2}/U_{r-1}$ in $SO_{2r}/U_r$.
\end{thm}

From Theorems \ref{known}--\ref{SO2rUr} we obtain the first part of Theorem \ref{classification}. 

\medskip
It remains to prove that any real hypersurface listed in Theorem \ref{classification} has isometric Reeb flow. We first describe $T_o\Sigma$ as a Lie triple system $\bigoplus_{\alpha \in \Delta_\cf} \CC u_\alpha = \cf \subset \cp$, using the root systems at the end of Section \ref{StcHss}. 

\noindent For $ \CC P^k \subset \CC P^r $ the root system is ($A_r$) and 
\[
\Delta_\cf = \{ e_1 - e_{\mu+1} : 1 \leq \mu \leq k\}.
\]
For $ G_k(\CC^r) \subset G_k(\CC^{r+1}) $ the root system is ($A_r$) and
\[
\Delta_\cf =  \{ e_\nu - e_{\mu+1}  : 1 \leq \nu \leq k \leq \mu \leq r-1\}.
\]
For $\CC P^{r-1} \subset G_2^+(\RR^{2r})$ the root system is ($D_r$) and
\[
\Delta_\cf =  \{ e_1 - e_{\mu+1} : 1 \leq \mu \leq r-1\}.
\]
For $ SO_{2r-2}/U_{r-1} \subset SO_{2r}/U_r $ the root system is ($D_r$) and
\[
\Delta_\cf =  \{e_\nu + e_\mu : 2 \leq \nu < \mu \leq r\}.
\]
In all cases $u_\delta$ is perpendicular to $\cf$ and thus $u_\delta/|u_\delta|$ is a unit normal vector of $\Sigma$ at $o$. Consider the geodesic $\gamma$ in $\bar{M}$ with $\gamma(0) =o$ and $\dot{\gamma}(0) = u_\delta/|u_\delta|$. The point $\gamma(t)$ is on the tube $M_t$ of radius $t$ around $\Sigma$. The Jacobi operator $\bar{R}_{u_\delta/|u_\delta|}$ is given by
\[
\bar{R}_{u_\delta/|u_\delta|}x = 
\begin{cases}
0 & ,\ {\rm if}\ x \in \RR u_\delta \oplus \bigoplus_{\alpha \in \Delta_M^+(0)} \CC u_\alpha\\
\frac{1}{2}x & ,\ {\rm if}\ x \in \bigoplus_{\alpha \in \Delta_M^+(1)} \CC u_\alpha   \\
2x & ,\ {\rm if}\ x \in \RR v_\delta.
\end{cases}
\]
We decompose $T_o\Sigma$ into $T_o\Sigma = T_o^0\Sigma \oplus T_o^1\Sigma$ with $T_o^0\Sigma = T_o\Sigma \cap \cp(0)$ and $T_o^1\Sigma = T_o\Sigma \cap \cp(1)$. Since $\Delta_M^+(0) \subset \Delta_\cf$, the normal space $\nu_o\Sigma$ decomposes into $\nu_o\Sigma = \nu_o^1\Sigma \oplus \CC u_\delta$ with $\nu_o^1\Sigma = \nu_o\Sigma \cap \cp(1)$.

Denote by $\gamma^\perp$ the parallel subbundle of $T\bar{M}$ along $\gamma$ defined by $\gamma^\perp_{\gamma(t)} = T_{\gamma(t)}\bar{M} \ominus {\mathbb R}\dot{\gamma}(t)$. Moreover, define the $\gamma^\perp$-valued tensor field $\bar{R}^\perp_{\gamma}$ along $\gamma$ by $\bar{R}^\perp_{\gamma(t)}X = \bar{R}(X,\dot{\gamma}(t))\dot{\gamma}(t)$.
Now consider the ${\rm End}(\gamma^\perp)$-valued differential equation $Y^{\prime\prime} + \bar{R}^\perp_{\gamma} \circ Y = 0$.
Let $D$ be the unique solution of this differential equation with initial values
\[
D(0) = \begin{pmatrix} I & 0 \\ 0 & 0 \end{pmatrix}\ ,\ D^\prime(0) = \begin{pmatrix} 0 & 0 \\ 0 & I \end{pmatrix},
\]
where the decomposition of the matrices is with respect to $\gamma^\perp_o = T_o\Sigma \oplus (\nu_o\Sigma \ominus \RR \dot\gamma(0))$
and $I$ denotes the identity transformation on the corresponding space. Then the shape operator $S(t)$ of $M_t$ with respect to $-\dot{\gamma}(t)$ is given by $S(t) = D^\prime(t)\circ D^{-1}(t)$
(see \cite{BCO16}, Section 10.2.3). If we now decompose $\gamma^\perp_o$ into
\[
\gamma^\perp_o = T_o^0F \oplus T_o^1F \oplus \nu_o^1F \oplus \RR v_\delta,
\]
we get by explicit computation that
\[
S(t) = \begin{pmatrix}
0 & 0 & 0 & 0 \\
0 & -\frac{1}{\sqrt{2}}\tan(\frac{1}{\sqrt{2}}t) & 0 & 0 \\
0 & 0 & \frac{1}{\sqrt{2}}\cot(\frac{1}{\sqrt{2}}t) & 0 \\
0 & 0 & 0 & \sqrt{2}\cot(\sqrt{2}t)
 \end{pmatrix}
\]
with respect to that decomposition. The principal curvatures of $M_t$ therefore are $0$, $-\frac{1}{\sqrt{2}}\tan(\frac{1}{\sqrt{2}}t)$, $\frac{1}{\sqrt{2}}\cot(\frac{1}{\sqrt{2}}t)$, $\sqrt{2}\cot(\sqrt{2}t)$ with multiplicities
\begin{itemize}
\item[(i)] $0$, $2k$, $2(r-k-1)$, $1$ for $\CC P^k \subset \CC P^r$;
\item[(ii)] $2(k-1)(r-k)$, $2(r-k)$, $2(k-1)$, $1$ for $G_k(\CC^r) \subset G_k(\CC^{r+1})$;
\item[(iii)] $2$, $2(r-2)$, $2(r-2)$, $1$ for $ \CC P^{r-1} \subset G_2^+(\RR^{2r})$;
\item[(iv)] $(r-3)(r-2)$, $2(r-2)$, $2(r-2)$, $1$ for $ SO_{2r-2}/U_{r-1} \subset SO_{2r}/U_r$.
\end{itemize}
Note that in case (i) the number of distinct principal curvatures is $2$ (for $k = 0$) or $3$ (for $k > 0$). In the other three cases there are four distinct principal curvatures. The corresponding principal curvature spaces are obtained by parallel translation of the corresponding subspaces in $\gamma_o^\perp$ along $\gamma$ from $o$ to $\gamma(t)$. By construction, $T_o^0\Sigma,T_o^1\Sigma,\nu_o^1\Sigma$ are $J$-invariant subspaces and span ${\mathcal C}_o$. Since $J$ is parallel, it follows that $S(t)\phi = \phi S(t)$. From Proposition \ref{isomcomm} we finally conclude that the Reeb flow on $M_t$ is an isometric flow. This finishes the proof of Theorem \ref{classification}

\end{document}